\documentclass{amsart}
\usepackage{mathrsfs}
\usepackage{mathtools} 
\usepackage{calligra}
\usepackage{amsthm}
\usepackage{amscd}
\usepackage{amssymb}
\usepackage{latexsym}
\usepackage{graphicx}
\usepackage{stmaryrd}
\usepackage{xypic}
\xyoption{curve}
\usepackage{nicefrac}
\usepackage{wasysym}
\usepackage{enumitem}
\usepackage{setspace}
\usepackage{ifthen}
\usepackage{verbatim}
\usepackage{eufrak}
\usepackage{relsize} 
\usepackage{tikz-cd}
\usepackage{stackengine} 
\usepackage{relsize} 
\usepackage[cal=boondoxo]{mathalfa}
\usepackage[margin=1.5in]{geometry}
\input{notation.tex} 

\newtheorem{thm}{Theorem}[section]
\newtheorem{cor}[thm]{Corollary}
\newtheorem{lem}[thm]{Lemma}
\newtheorem{prop}[thm]{Proposition}

\newtheorem*{lem*}{Lemma}

\theoremstyle{definition}
\newtheorem{defn}[thm]{Definition}

\newtheorem{eg}[thm]{Example}
\newtheorem{rem}[thm]{Remark}
\usepackage[disable]{todonotes}

\newcommand{\chains}{\mathfrak{ch}}
\newcommand{\pls}{\raisebox{.4\height}{\scalebox{.6}{+}}}
\newcommand{\mins}{\raisebox{.4\height}{\scalebox{.8}{-}}}

\newcommand{\firstvertexmap}{\gamma^{\mins}}
\newcommand{\lastvertexmap}{\gamma^{\pls}}

\newcommand{\adjointfirstvertexmap}{\hat{\gamma}^{\mins}}
\newcommand{\adjointlastvertexmap}{\hat{\gamma}^{\pls}}
\newcommand{\lpcohomology}[2]{H^{#1}_{\ell_{#2}}}
\newcommand{\boundedcohomology}[1]{\lpcohomology{#1}{\infty}}
\newcommand{\ex}{\mathfrak{ex}}
\newcommand{\homotopicalcat}[1]{\ifthenelse{1=#1}{\mathscr{H}}{\ifthenelse{2=#1}{\mathscr{B}}{\mathscr{C}}}}

\newcommand{\C}{\mathbb{C}}
\newcommand{\cnerve}{\nerve_\BOX}
\newcommand{\snerve}{\nerve_\DEL}

\newcommand{\METRIC}{\mathscr{M}}
\newcommand{\TOP}{\mathscr{T}}
\newcommand{\ABGROUPS}{\mathscr{A}}
\newcommand{\UNIFORM}{\mathscr{U}}

\newcommand{\CONNECTIVECHAINCOMPLEXES}{\mathcal{Ch}_{\pls}}
\newcommand{\SETS}{\mathscr{S}}

\newcommand{\geometricrealization}[1]{|#1|_{\sphericalangle}}
\newcommand{\uniformrealization}[1]{|#1|_{\infty}}

\theoremstyle{plain}
\newtheorem*{lem:barycentric}{Lemma 2.2.9, \cite{waldhausen2013spaces}}
\newtheorem*{thm:monoidal.forgetful.functor}{Theorem 22, \S III \cite{isbell1964uniform}}
\newtheorem*{thm:representation}{Theorem \ref{thm:representation}}
\newtheorem*{thm:closed.semi-uniform.product}{Theorem 26, \S III \cite{isbell1964uniform}}
\newtheorem*{thm:semi-uniform.cartesian}{Theorem 24, \S III \cite{isbell1964uniform}}
\newtheorem*{thm:compact.uniform.spaces}{Corollary 24, \S II \cite{isbell1964uniform}}
\newtheorem*{cor:embedding}{Corollary \ref{cor:embedding}}
\newtheorem*{cor:bounded.geometry}{Corollary \ref{cor:bounded.geometry}}
\newtheorem*{thm:kan.cubical.groups}{Theorem 2.1, \cite{tonks1992cubical}}
\newtheorem*{thm:triangulations}{Theorem 1.1, \cite{bowditch2020bilipschitz}}
\newtheorem*{thm:cubical.cohomology.representable}{Theorem \ref{thm:cubical.cohomology.representable}}
\newtheorem*{cor:singular.cohomology.representable}{Corollary \ref{cor:singular.cohomology.representable}}
\newtheorem*{lem:collapse.star}{Lemma 6.11, \cite{krishnan2015cubical}}
\newtheorem*{lem:star.flower}{Lemma 6.12, \cite{krishnan2015cubical}}
\newtheorem*{lem:natural.retractions}{Lemma 6.13, \cite{krishnan2015cubical}}
\newtheorem*{lem:qt.cocontinuous}{Lemma 4.2, \cite{krishnan2015cubical}}
\newtheorem*{thm:dold-kan}{Theorem 14.8.1, \cite{huebschmann2012ronald}}

\title{The uniform homotopy category}
\author{Sanjeevi Krishnan and Crichton Ogle}

\begin{document}
\begin{abstract}
  This paper gives a uniform-theoretic refinement of classical homotopy theory.  
  Both cubical sets (with connections) and uniform spaces admit classes of weak equivalences, special cases of classical weak equivalences, appropriate for the respective Lipschitz and uniform settings.  
  Cubical sets and uniform spaces admit the additional compatible structures of categories of (co)fibrant objects.  
  A categorical equivalence between classical homotopy categories of cubical sets and spaces lifts to a full and faithful embedding from an associated \textit{Lipschitz homotopy category} of cubical sets into an associated \textit{uniform homotopy category} of uniform spaces.
  Bounded cubical cohomology generalizes to a representable theory on the Lipschitz homotopy category.
  Bounded singular cohomology on path-connected spaces generalizes to a representable theory on the uniform homotopy category.  
  Along the way, this paper develops a cubical analogue of Kan's $\mathrm{Ex}^\infty$ functor and proves a cubical approximation theorem for uniform maps.  
\end{abstract}
\maketitle
\tableofcontents
\addtocontents{toc}{\protect\setcounter{tocdepth}{1}}

\section{Introduction}
Loosely, $\ell_p$ cohomology $\lpcohomology{*}{p}(X;\pi)$ is a variant of classical cohomology on geometric objects $X$ of some sort with coefficients in normed modules $\pi$, obtained by restricting the cochains in the definition to have finite $\ell_p$-norm induced from the geometry on $X$ and norm on $\pi$.
In this sense, the case $p=\infty$ covers variants of \textit{bounded cohomology} \cite{gromov1982volume,ogle2010polynomially} (cf. \cite{gersten1996note}.)
An inclusion of cochain complexes induces a \textit{comparison map}
\begin{equation}
  \label{eqn:comparison.map}
  \lpcohomology{*}{p}(X;\pi)\ra H^*(X;\pi)
\end{equation}
to ordinary cohomology, which measures some degree to which topology constrains geometry.
For example, hyperbolicity for a group $G$ can be characterized by the surjectivity of (\ref{eqn:comparison.map}) from bounded group cohomology $\boundedcohomology{2}(G;\pi)$ for suitable choices of normed coefficients $\pi$ 
\cite[Theorem 3]{mineyev2002bounded}.

This paper lays down some foundations for a homotopy theory of geometric objects in nature, such as Riemannian manifolds or groups with word length functions, convenient for comparisons of $\ell_p$-cohomology theories with one another and with ordinary cohomology.  
Modern foundations for classical homotopy are convenient for comparisons between ordinary cohomology theories.
For example, isomorphisms between ordinary cubical, simplicial, and topological cohomology follows from categorical equivalences between classical homotopy categories of cubical sets, simplicial sets, and spaces \cite[Chapter II \S 3]{quillen2006homotopical}, \cite{cisinskiprefaisceaux}.  
While $\ell_p$ simplicial and $\ell_p$ de Rham theories coincide on suitably triangulated Riemannian manifolds \cite{gol1988rham}, there does not exist in the literature an analogous equivalence of homotopy categories appropriate for the $\ell_p$ setting.
For another example, the natural comparison map $H^*(X;\Z)\ra H^*(X;\Z_p)$ is bijective when the corresponding map of representing spectra is a weak equivalence after smashing with $X$.  
While represented at the chain level (eg. \cite{MR3010108}), (\ref{eqn:comparison.map}) is not represented in the literature as a morphism in a homotopy category of intuitively geometric objects [Example \ref{eg:non-representability}]. 
This paper narrows the gap between the kinds of homotopical tools available for ordinary cohomology with analogous homotopical tools for $\ell_p$ cohomology.  

Classical combinatorial homotopy theory can be refined so as to take into account geometry implicit in the combinatorics.   
\textit{Cubical sets} are formal colimits of abstract hypercubes.
There exists a model structure on cubical sets $\CUBICAL\SETS$ analogous and equivalent to the usual model structures on simplicial sets and spaces \cite{cisinskiprefaisceaux}.
Classical weak equivalences of cubical sets are in some sense cubical homotopy equivalences up to infinite subdivision [Corollary \ref{cor:classical.we}].  
\textit{Lipschitz weak equivalences}, special cases of classical weak equivalences, can be defined as cubical homotopy equivalences up to \textit{finite} subdivision [Definition \ref{defn:lipschitz.we}].
The Lipschitz weak equivalences, at least between connected cubical sets, turn out to be exactly those cubical functions which induce Lipschitz homotopy equivalences between geometric realizations [Corollary \ref{cor:cubical.we}], as a consequence of cubical approximation.   
Lipschitz weak equivalences and monos give cubical sets $\CUBICAL\SETS$ the structure of a category of cofibrant objects, whose fibrant objects are the Kan complexes [Proposition \ref{prop:cofibrant.objects}].
The \textit{Lipschitz homotopy category} $h_\infty(\CUBICAL\SETS)$, the localization of cubical sets $\CUBICAL\SETS$ by Lipschitz weak equivalences, contains the classical homotopy category $h(\CUBICAL\SETS)$ of cubical sets up to equivalence as the full subcategory of Kan complexes.  

The Lipschitz homotopy category is a convenient setting for cubical $\ell_p$ cohomology.  
Cubical $\ell_p$ cohomology is weak Lipschitz invariant on connected, \textit{sharp} [Definition \ref{defn:sharp.replacement}] cubical sets [Proposition \ref{prop:lipschitz.invariance}], including Kan complexes.
A cubical homomorphism 
\begin{equation}
  \label{eqn:representation.of.comparison.map}
  C_\infty(\pi,n)\ra C(\pi,n)
\end{equation}
between cubical commutative semigroups can be constructed between two $n$-fold deloopings of $\pi$, only the former of which takes into account the norm on $\pi$ and only the latter of which is Kan.  
This homomorphism represents the comparison map (\ref{eqn:comparison.map}) for bounded cohomology $\lpcohomology{*}{\infty}$ on connected Kan cubical sets like $C(G,1)$, including bounded group cohomology [Example \ref{eg:bounded.group.cohomology}].

\newcommand{\THMCubicalCohomologyRepresentable}{
  There exist dotted isomorphisms in the diagram
  \begin{equation*}
    \xymatrix{
	**[l]\boundedcohomology{n}(B;\pi)\ar[rr] & & **[r]H^n(B;\pi)\\
	**[l][B,C_\infty(\pi,n)]\ar@{.>}[u]\ar[rr]_{[B,C_\infty(\pi,n)\ra C(\pi,n)]} & & **[r][B,C(\pi,n)]\ar@{.>}[u]_{\cong}
	}
  \end{equation*}
  natural in cubical sets $B$ that are either finite or both connected and Kan, and normed Abelian groups $\pi$, making the entire square commute.
  Here $[-,-]$ denotes a hom-set in $h_\infty(\CUBICAL\SETS)$, the top horizontal arrow is the comparison map, and the bottom horizontal arrow is induced by the quotient cubical semigroup homomorphism $C_\infty(\pi,n)\ra C(\pi,n)$.}

\begin{thm:cubical.cohomology.representable}
  \THMCubicalCohomologyRepresentable{}
\end{thm:cubical.cohomology.representable}

Classical topological homotopy theory can be refined so as to take into account geometry explicitly given as extra structure on a space.   
\textit{Uniform spaces} \cite{isbell1964uniform} axiomatize the structure preserved by uniformly continuous maps and therefore the structure that controls $\ell_p$ cohomology \cite{elek1998coarse,pansu2007lp}.
Classical weak equivalences of spaces can be characterized in some sense by probes from topological realizations of cubical sets.
\textit{Uniform weak equivalences} can be analogously characterized [Definition \ref{defn:uniform.we}] by probes from \textit{uniform realizations}, uniform refinements of topological realizations, of cubical sets.  
Uniform weak equivalences and uniform analogues of Hurewicz fibrations give uniform spaces $\UNIFORM$ the structure of a category of fibrant objects [Proposition \ref{prop:fibrant.uniform.spaces}], whose cofibrant objects include all finite-dimensional simplicial complexes [Corollary \ref{cor:cofibrant.manifolds}] and spaces of finite CW type equipped with their fine uniformities [Corollary \ref{cor:spaces}].  
The \textit{uniform homotopy category} $h_\infty\UNIFORM$, the localization by uniform weak equivalences, contains the full subcategory of finite CW complexes as a full subcategory.  

The Lipschitz homotopy category actually fully embeds into the uniform homotopy category.
The equivalence between classical homotopy categories of simplicial sets and spaces follows from a simplicial approximation theorem for finite simplicial sets, the closure of classically acyclic cofibrations under transfinite composition, and the preservation of classically acyclic cofibrations by topological realization \cite{curtis1971simplicial,jardine2004simplicial}.
On one hand, this transfinite closure fails in the uniform setting.
On the other hand, the finiteness restriction for simplicial approximation \cite[Theorem 2C.1]{hatcher2005algebraic} is unnecessary in the uniform setting. 
We therefore obtain a cubical approximation theorem for uniform maps between general uniform realizations [Lemma \ref{lem:adjoint.approximation}], whose proof borrows techniques from an analogous result for \textit{directed maps} between compact \textit{directed realizations} \cite{krishnan2015cubical}.
And uniform realization preserves acyclic cofibrations [Lemma \ref{lem:extensions}].
A consequence is the following full inclusion of the Lipschitz homotopy category into the uniform homotopy category.

\newcommand{\COREmbedding}{
  Uniform realization induces a dotted full and faithful embedding making
  \begin{equation*}
    \xymatrix{
	**[l]h_\infty(\CUBICAL\SETS)\ar[d]\ar@{.>}[r] & **[r]h_\infty\UNIFORM\ar[d]\\
	  **[l]h(\CUBICAL\SETS)\ar[r]_{\simeq} & **[r]h\TOP,
	}
  \end{equation*}
  commute, where the left vertical arrow is induced by the identity on $\CUBICAL\SETS$, the right vertical arrow is induced by the forgetful functor $\UNIFORM\ra\TOP$, and the bottom horizontal functor is equivalence of classical homotopy categories induced by topological realization.
}

\begin{cor:embedding}
  \COREmbedding{}
\end{cor:embedding}

The uniform homotopy category is a convenient setting for a uniform-theoretic generalization of bounded singular cohomology.  
The uniform realization $K_\infty(\pi,n)\ra K(\pi,n)$ of (\ref{eqn:representation.of.comparison.map}) represents the comparison map for bounded singular cohomology on spaces.
The latter theory, defined below in terms of the right adjoint $\sing_{\BOX}$ to topological realization of cubical sets, is equivalent to the usual theory analogously defined in terms of semi-simplicial sets [Example \ref{eg:bounded.singular.cohomology}].

\newcommand{\CORSingularCohomologyRepresentable}{
  There exist dotted isomorphisms in the diagram
  \begin{equation*}
    \xymatrix{
	**[l]\boundedcohomology{n}(\sing_{\BOX}X;\pi)\ar[rr] & & **[r]H^n(\sing_{\BOX}X;\pi)\\
	**[l][X,K_\infty(\pi,n)]\ar@{.>}[u]\ar[rr]_{[X,K_\infty(\pi,n)\ra K(\pi,n)]} & & **[r][X,K(\pi,n)]\ar@{.>}[u]_{\cong}
	}
  \end{equation*}
  natural in path-connected spaces $X$ and normed Abelian groups $\pi$, making the entire square commute.
  Here $[-,-]$ denotes a hom-set in $h_\infty\UNIFORM$, $X$ is regarded in the bottom row as a uniform space with its fine uniformity, the top horizontal arrow is the comparison map, and the bottom horizontal arrow is induced by the quotient uniform map $K_\infty(\pi,n)\ra K(\pi,n)$.}

\begin{cor:singular.cohomology.representable}
  \CORSingularCohomologyRepresentable{}
\end{cor:singular.cohomology.representable}

An open problem is to formalize the relationship between the uniform-theoretic generalization $[-,K_\infty(\pi,n)]$ of bounded singular cohomology and \u{C}ech-like theories based on uniform covers \cite{andre2012uniform,bahauddin1973homology}, of interest in applications to differential equations \cite{andre2012uniform}.  
Another open problem is to recover information about the seminorm on bounded cohomology, important for applications to cubical volume \cite{loeh2015cubical,gromov1982volume}, from uniform homotopy theory alone.  
Bounded group cohomology $\boundedcohomology{n}(G;\pi)=[K(G,1),K_\infty(\pi,n)]$ admits variants by encoding extra geometry on a group $G$, like a word-length function, by a Lipschitz weak type with the same underlying classical weak type as $K(G,1)$ [Examples \ref{eg:non-kan.BG}, \ref{eg:non-kan.BG.cocycle}].
It is left as an open question whether such variants subsume \textit{polynomially bounded group cohomology} \cite{gol1988rham}, from which the comparison map to ordinary group cohomology is of interest in certain approaches to the Novikov Conjecture \cite{connes1988hyperbolic,ogle2010polynomially}.
Another open question is whether cubical sets can be replaced with simplicial sets in the whole theory, despite technical challenges in the simplicial setting [Remark \ref{rem:simplicial.disadvantages}].

\subsection{Outline}
Section \S\ref{sec:geometry} recalls and compares the geometry of cubical sets, pseudometric spaces, and uniform spaces.
While most of this section simply recalls existing theory in the literature, a kind of mapping complex for Lipschitz homotopy theory, and in particular an analogue of Kan's $\mathrm{Ex}^\infty$ functor \cite{kan1957css}, is introduced [Definition \ref{defn:lipschitz.mapping.complex}].
Section \S\ref{sec:homotopy} formalizes and compares homotopy theories for cubical sets, pseudometric spaces, and uniform spaces.
Section \S\ref{sec:cohomology} formalizes and compares cohomology theories for cubical sets and uniforms spaces.
Appendix \S\ref{sec:fibrant.objects} recalls basic definitions and examples of \textit{categories of (co)fibrant objects} \cite{brown1973abstract}.
Appendix \S\ref{sec:simplicial} presents simplicial sets as technically convenient geometric objects that are intermediate in rigidity between cubical sets and pseudometric spaces.  

\subsection{Conventions}
We fix the following conventions throughout the paper.  
The unit interval $\I=[0,1]$ will either mean the metric space when equipped with the usual Euclidean metric or simply the underlying space.  
Let $k,m,n,i,j$ denote natural numbers. 
Let $p$ denote either a positive real number greater than or equal to $1$ or $\infty$.  
Write $[0,\infty]$ for the disjoint union of the non-negative real numbers with the singleton containing $\infty$; $[0,\infty]$ will be regarded as a totally ordered set under the usual ordering and a commutative monoid under the usual addition.
In order to emphasize the connection with other $\ell_p$-cohomology theories, $\boundedcohomology{*}$ will refer to variants of \textit{bounded cohomology} as opposed to a somewhat dual theory often similarly notated in the literature \cite{gersten1996note}.  
And both $\ell_p$ cohomology for $p<\infty$ as well as bounded cohomology, in all of their variants, will refer to just the Abelian group without the induced seminorm.

\subsubsection{Relations}
Functions $f$ are the data of a domain $X$, a codomain $Y$, and a subset of $X\times Y$, the graph of $f$, satisfying certain properties.  
More generally take a \textit{relation} $E$ to mean the data of a \textit{domain} set $X$, \textit{codomain} set $Y$, and a subset $\graph{(E)}\subset X\times Y$.
Write 
$$E:X\relation Y$$ 
to denote a relation $E$ with domain $X$ and codomain $Y$.
For $E:X\relation Y$ and $x\in X$, let $E[x]$ denote the set of all $y\in Y$ with $(x,y)\in\graph{(E)}$.  
For each relation $E:X\relation Y$, write $E^{-1}$ for the relation $Y\relation X$ with $x\in E^{-1}[y]$ if and only if $y\in E[x]$.  
For a pair of relations $E_1:X\relation Y$ and $E_2:Y\relation X$, write $E_2\circ E_1$ for the relation $X\relation Z$ with $(x,z)\in\graph{(E_2\circ E_1)}$ if and only if there exists $y\in E_1[x]$ such that $z\in E_2[y]$.
A relation $E:X\relation X$ is \textit{reflexive} if $\graph{(E)}$ contains the diagonal of $X$.  

\subsubsection{Categories}
Notate special categories as follows.
\vspace{.1in}\\
\begin{tabular}{rl}
	$\SETS$ & sets and functions\\
	$\METRIC$ & pseudometric spaces and $1$-Lipschitz maps\\
	$\TOP$ & (topological) spaces and (continuous) maps\\
	$\UNIFORM$ & uniform spaces and uniform maps\\
	$\CUBICAL\SETS$ & cubical sets and cubical functions\\
	$\ABGROUPS$ & Abelian groups and group homomorphisms\\
	$\CONNECTIVECHAINCOMPLEXES$ & connective chain complexes of Abelian groups and chain homomorphisms\\
	[.1in]
\end{tabular}

Regard $\CONNECTIVECHAINCOMPLEXES$ as $\ABGROUPS$-enriched so that $\CONNECTIVECHAINCOMPLEXES(A,B)$ is the Abelian group of chain maps $A\ra B$ under degree-wise, element-wise addition.
Let $\Sigma^n$ denote the functor $\ABGROUPS\ra\CONNECTIVECHAINCOMPLEXES$ naturally sending each Abelian group $\pi$ to the chain complex concentrated in degree $n$ with $(\Sigma^n\pi)_n=\pi$.
Write $[n]$ for the ordinal $\{0<1<\cdots<n\}$. 
Write $\id_o$ for the identity morphism of an object $o$ in a given category.
Write $\star$ for the terminal object in a given complete category.
Let $\hookrightarrow$ denote an inclusion of some sort, such as between presheaves.  
Let $\pi_0$ denote the functor $\TOP\ra\SETS$ naturally sending a space to its set of path-components.
Write $\otimes$ for the tensor in a given monoidal category or the tensor $\otimes:\cat{1}\times\mathscr{V}\ra\cat{1}$ of a $\mathscr{V}$-category $\cat{1}$.   
Take a \textit{monoidal adjunction} to mean an adjunction with left adjoint strict monoidal.  
Write $\vee_{L}$ for the binary supremum operator $L^2\ra L$ in a sup-semilattice $L$.  

\subsubsection{Semigroups}
A \textit{(commutative) semigroup} is a set equipped with an associative (and commutative) multiplication.  
A \textit{semigroup homomorphism} is a function between semigroups preserving the multiplication.  
Denote the multiplication of an Abelian group additively.
Let $0$ denote the identity of an Abelian group.  

\subsubsection{Seminorms}
Take a \textit{seminorm} on an Abelian group $\pi$ to mean a function 
$$\|-\|:\pi\ra\R_{\geqslant 0}$$
satisfying the following relations for $g,h\in \pi$:
\begin{align*}
  \|0\|&=0\\
  \|g\|&=\|-g\|\\
  \|g+h\|&\leqslant\|g\|+\|h\|
\end{align*}

A seminorm $\|-\|$ on an Abelian group $\pi$ is a \textit{norm} if $g=0$ whenever $\|g\|=0$.  
Take a \textit{(semi)normed Abelian group} $\pi$ to mean an Abelian group $\pi$ equipped with a (semi)norm on it, which we denote by $\|-\|$.
Write $\ell_p(X,\pi)$ for the Abelian group of functions $c:X\ra\pi$ to a seminormed Abelian group $\pi$ for which $\sum_{x\in X}\|c(x)\|^p<\infty$ in the case $p<\infty$ and $\sup_{x\in X}\|c(x)\|<\infty$ in the case $p=\infty$.
A \textit{bounded homomorphism} is a homomorphism $\phi:G\ra H$ of seminormed Abelian groups for which $\phi\in\ell_\infty(G,H)$.

\addtocontents{toc}{\protect\setcounter{tocdepth}{2}}

\section{Geometry}\label{sec:geometry}
Geometry can be implicitly encoded in the combinatorics of a presheaf, explicitly as the (extended pseudo)metric of a (pseudo)metric space, or abstractly as the uniformity of a uniform space.
The reader is referred elsewhere for a detailed treatment of cubical sets \cite{grandis2003cubical}, pseudometric spaces \cite{goubault2019directed,lawvere1973metric}, and uniform spaces \cite{isbell1964uniform}.  
Section \S\ref{subsec:cubical.sets} recalls the basic theory of cubical sets and introduces some constructions [Definitions \ref{defn:sharp.replacement}, \ref{defn:lipschitz.mapping.complex}] needed to later define \textit{Lipschitz weak equivalences}. 
Section \S\ref{subsec:pseudometric.spaces} recalls the relevant existing theory of pseudometric spaces.
Section \S\ref{subsec:uniform.spaces} recalls the relevant existing theory of uniform spaces and fixes a convenient category for them.
Section \S\ref{subsec:comparisons} compares categories of cubical sets, pseudometric spaces, uniform spaces, and spaces along forgetful and realization functors.

\subsection{Cubical}\label{subsec:cubical.sets}
We first fix the site over which cubical objects are defined to be the minimal variant in the literature containing coconnections of one type.  
We then recall some mostly standard notation and results concerning cubical sets thus defined, usually referred to in the literature as \textit{cubical sets with connections}.   
Finally, we extend a subdivision operator on traditionally defined cubical sets and subsequently construct \textit{Lipschitz mapping complexes}.

\subsubsection{The site}
Write $\delta_\pm$ for the monotone functions $[0]\ra[1]$ sending $0$ to $\half\pm\half$.  
Let $\BOX_1$ denote the category whose objects are $[0],[1]$ and whose morphisms are generated by $\delta_{\pm},[1]\ra[0]$.
Let $\BOX_2$ denote the category whose objects are $[0],[1],[1]^2$ and whose morphisms are generated by $\delta_{\pm},[1]\ra[0]$, and $\vee_{\boxobj{2}}:\boxobj{2}\ra[1]$.
Write $\BOX$ for the submonoidal category of the Cartesian monoidal category of small categories and functors between them generated by $\BOX_2$.  
In contract, the traditional, most common, and minimal variant of $\BOX$ is the submonoidal category of $\BOX$ generated by $\BOX_1$.  
Write $\boxobj{n}$ for the $n$-fold tensor product of $[1]$ in $\BOX$.

\begin{eg}
  The objects of $\BOX$ are the sup-semilattices 
  $$[0],[1],\boxobj{2},\boxobj{3},\cdots$$
\end{eg}

\begin{eg}
  The morphisms of $\BOX$ are sup-semilattice homomorphisms of the form
  $$\boxobj{m}\ra\boxobj{n}.$$
\end{eg}

When $n$ is understood, for each $1\leqslant i\leqslant n$ let
\begin{align*}
	\delta_{\pm i}&=\boxobj{i}\otimes\delta_{\pm}\otimes\boxobj{n-i}:\boxobj{n}\ira\boxobj{n+1}
\end{align*}

\subsubsection{Cubical objects}
Write $\CUBICAL\cat{1}$ for the category of functors $\OP{\BOX}\ra\cat{1}$ and natural transformations between them.  
Regard $\CUBICAL\SETS$ as closed monoidal category with tensor the cocontinuous extension of the tensor on $\BOX$ along the Yoneda embedding.  
\textit{Cubical sets} and \textit{cubical functions} are, respectively, the objects and morphisms of $c\SETS$. 
\textit{Cubical (commutative semi)groups} and \textit{cubical homomorphisms} are, respectively, the objects and morphisms of $\CUBICAL\cat{1}$ for the case $\cat{1}$ the category of (commutative semi)groups and semigroup homomorphisms between them.
We often treat cubical (commutative semi)groups as cubical sets by implicitly composing them with the forgetful functor to $\SETS$.  

\begin{eg}
  The representable cubical sets $\BOX[0],\BOX[1],\BOX\boxobj{2},\ldots$ are defined by
  $$\BOX\boxobj{n}(-)=\BOX(-,\boxobj{n}):\OP{\BOX}\ra\SETS.$$
\end{eg}

We regard $\CUBICAL\SETS$ as closed monoidal with tensor product $\otimes$ extending the tensor product on $\BOX$ along the functor $\BOX^2\ra(\CUBICAL\SETS)^2$ induced from the Yoneda Embedding $\BOX[-]:\BOX\ira\CUBICAL\SETS$.  
We henceforth $\CUBICAL\SETS$ as self-enriched with $\CUBICAL\SETS(B,-)$ right adjoint to $B\otimes-$.  
There exist monic inclusions, natural in cubical sets $B,C$, of the form 
$$B\otimes C\ira B\times C.$$

Fix a cubical set $C$. 
Write $C_n$ for $C(\boxobj{n})$. 
An \textit{$n$-cube} in $C$ is an element in $C_n$; a \textit{vertex} is a $0$-cube. 
Write $\theta_*$ for the image of $\theta\in C_n$ under the natural bijection 
$$C_n\cong c\SETS(\BOX\boxobj{n},C)_0.$$

For a given cubical set $C$ and natural number $n$, let
$$d_{\pm i}=C(\delta_{\pm i}:\boxobj{n}\ra\boxobj{n+1}):C_{n+1}\ra C_n.$$
Define $sC_{-1}=\varnothing$ and for each $n$ define $sC_n$ to be the subset of $C_{n+1}$ consisting of elements in the image of a function of the form $C(\sigma)$ for $\sigma$ a surjective $\BOX$-morphism from $\boxobj{n+1}$.  
An $(n+1)$-cube in $C$ is \textit{non-degenerate} if it does not lie in $sC_n$.  
The \textit{dimension} of $C$ is the infimum over all natural numbers $n$ for which $C_{n+1}=sC_n$.
A cubical set $C$ is \textit{$n$-coskeletal} if every cubical function $A\ra C$ from an $n$-dimensional cubical set $A$ uniquely extends along an inclusion $A\ira B$ of subpresheaves such that $A_n=B_n$.
Call $C$ \textit{finite} if it is finite-dimensional and only finitely has many non-degenerate $n$-cubes for each $n$.  

\begin{eg}
  A $1$-dimensional cubical set $C$ is the data of the directed graph with \ldots
  \begin{enumerate}
	\item \ldots vertex set $C_0$
	\item \ldots edge set $C_1-sC_0$
	\item \ldots source and target functions the restrictions of $d_{-1},d_{+1}$ to functions $C_1-sC_0\ra C_0$
  \end{enumerate}
  In particular, the finite cubical set $\BOX[1]$ represents the directed graph $\bullet\ra\bullet$.  
\end{eg}

Write $\partial\BOX\boxobj{n}$ for the maximal subpresheaf of $\BOX\boxobj{n}$ with
$$\id_{\boxobj{n}}\notin(\partial\BOX\boxobj{n})_n.$$

Write $\sqcup^{\pm i}\BOX\boxobj{n}$ for the maximal subpresheaf of $\BOX\boxobj{n}$ with
$$\delta_{\pm i}\notin(\sqcup^i\boxobj{n})_{n-1}.$$

The inclusion $\sqcup^{\pm i}\BOX\boxobj{n}\ira\BOX\boxobj{n}$ models the inclusion of an empty box without a lid into a closed and solid box.
A cubical set $C$ is \textit{Kan} if every cubical function of the form $\sqcup^{\pm i}\boxobj{n}\ra C$ extends along the inclusion $\sqcup^{\pm i}\boxobj{n}\ira\BOX\boxobj{n}$.  

\begin{thm:kan.cubical.groups}
  Each cubical group is Kan.
\end{thm:kan.cubical.groups}

Define $\pi_0C$ for the following coequalizer diagram natural in cubical sets $C$:
$$\xymatrix{C_1\ar@<.7ex>[rr]^{d_{-1}}\ar@<-.7ex>[rr]_{d_{+1}} & & C_0 \ar[r] & \pi_0C}$$
A cubical set $C$ is \textit{connected} if $\pi_0C$ is a singleton.
Write $\cnerve\cat{1}$ for the \textit{cubical nerve}, natural in small categories $\cat{1}$, defined as the cubical set naturally sending each $\BOX$-object $\boxobj{n}$ to the set of functors $\boxobj{n}\ra\cat{1}$.  
Write $\Tau_1$ for the monoidal left adjoint to $\cnerve$, regarded as a functor from the category of small categories and functors between them to $\CUBICAL\SETS$.  

\begin{eg}
  We can make the natural identification $\Tau_1\BOX\boxobj{n}=\boxobj{n}$.
\end{eg}

\begin{eg}
  For each group $G$, $\cnerve G$ is Kan because each inclusion of the form
  $$\sqcup^i\boxobj{n}\ira\BOX\boxobj{n}$$
  induces an isomorphism of groupoids after applying the left adjoint to the restriction of $\cnerve$ to groupoids, the composite of $\Tau_1$ with groupoidification.  
\end{eg}

Write $(-)^\sharp\dashv(-)_\sharp$ for the adjunction
$$(-)^\sharp:\CUBICAL\SETS\lras\CUBICAL\SETS:(-)_\sharp$$
with $(\BOX\boxobj{n})^\sharp$ the cubical nerve $\cnerve\boxobj{n}$ natural in $\BOX$-objects $\boxobj{n}$.
The restriction of the cubical nerve functor to $\BOX$ sends tensor products to binary products.  
Therefore $(-)^\sharp$ sends tensor products to binary products.
Regard the natural cubical function $C\ra C^\sharp$ induced by inclusions of the form $\BOX\boxobj{n}\ira\cnerve\boxobj{n}$ as an inclusion of subpresheaves.  
Recall that some of the properties of fibrant objects in a model structure extend to more general \textit{sharp objects} \cite{rezk1998fibrations}.  
We label certain cubical sets \textit{sharp} if they function as analogues of sharp objects with respect to a Lipschitz homotopy theory later defined.    

\begin{defn}
  \label{defn:sharp.replacement}
  Call a cubical set $C$ \textit{sharp} if $C\ira C^\sharp$ admits a retraction.  
\end{defn}

Each cubical set of the form $C^\sharp$ is sharp by the following lemma, whose proof uses intermediate simplicial constructions and is therefore deferred until the end of \S\ref{sec:simplicial}.  

\begin{lem}
  \label{lem:sharp.sharp}
  For each cubical set $C$, the inclusion
  $$C^\sharp\ira C^{\sharp\sharp}$$
  admits a retraction natural in cubical sets $C$.  
\end{lem}

We introduce some general notation for the support of a point in some construction on a cubical set, like a realization or a subdivision.
Consider a functor of the form $F:\CUBICAL\SETS\ra\cat{1}$.
For each cubical set $C$ and $\cat{1}$-morphism $\zeta$ to $FC$ which factors through $F(B\ira C)$ for some image $B$ of a representable, write $\support_F(\zeta,C)$ for the terminal $(\BOX[-]/C)$-object $\sigma:\BOX\boxobj{n}\ra C$ for which $\zeta$ factors through $F(\im\,\sigma\ira C)$.  

\subsubsection{Subdivisions}
We extend a subdivision operator on traditionally defined cubical sets \cite{jardine2002cubical} for cubical sets in the sense of this paper.  
We then build a Lipschitz mapping complex, and in particular a cubical analogue of Kan's $\mathrm{Ex}^\infty$ functor, in terms of the right adjoint to this subdivision operator.  
We then conclude with a cubical analogue of a classical result in simplicial theory that double barycentric subdivision factors through a polyhedral complex.  
To start, write $\sd\,\BOX[1]$ for the maximum $1$-dimensional subpresheaf of the cubical nerve of the poset of non-empty subsets of $[1]$ ordered by inclusion.  

\begin{eg}
  The cubical set $\sd\,\BOX[1]$ can be depicted as the directed graph
  $$\bullet\ra\bullet\la\bullet.$$
\end{eg}

Define $\sd\,\BOX\boxobj{n}=(\sd\,\BOX[1])^{\otimes n}$.  
The vertices of $\sd\,\BOX\boxobj{n}$ are the $n$-fold products of non-empty subsets of $[1]$, intervals in the poset $\boxobj{n}$.  
Each $\BOX$-morphism $\phi:\boxobj{m}\ra\boxobj{n}$, interval-preserving because it is a Cartesian monoidal product of interval-preserving monotone functions in $\BOX_2$, induces a function $(\sd\BOX[\phi])_0:(\sd\,\BOX\boxobj{m})_0\ra(\sd\,\BOX\boxobj{n})_0$.  
The function $(\sd\BOX[\phi])_0$ extends, necessarily uniquely, to a cubical function
$$\sd\BOX[\phi]:\sd\BOX\boxobj{m}\ra\sd\BOX\boxobj{n}$$
for $\phi=\delta_{\pm},\vee_{\boxobj{2}}$ by direct verification.  
Thus $\sd\BOX[-]$ defines a functor $\BOX_2\ra\CUBICAL\SETS$ which extends uniquely to a cocontinuous monoidal endofunctor on $\CUBICAL\SETS$, which we write as $\sd$.  
Write $\ex$ for the right adjoint in the monoidal adjunction
$$\sd:\CUBICAL\SETS\lras\CUBICAL\SETS:\ex.$$

Define natural transformations $\gamma^{\pm}:\sd\ra\id_{\CUBICAL\SETS}$, $\adjointfirstvertexmap,\adjointlastvertexmap:\id_{\CUBICAL\SETS}\ra\ex$ so that 
$$(\gamma^{\pm}_{\BOX[1]})_0(\{0,1\})=\half\pm\half$$ 
$\gamma^{\pm}:\sd\ra\id_{\CUBICAL\SETS}$ are monoidal, and $\adjointfirstvertexmap,\adjointlastvertexmap$ are the respective component-wise adjoints to $\firstvertexmap,\lastvertexmap$.

\begin{eg}
  The cubical functions 
  $$\gamma^{\pm}_{\BOX[1]}:\sd\,\BOX[1]\ra\BOX[1]$$
  map the middle vertex in $\bullet\ra\bullet\la\bullet$ to the respective left, right vertices in $\bullet\ra\bullet$.
\end{eg}

\begin{lem}
  \label{lem:dilation.order}
  For all cubical sets $C$ and $v\in(\sd\,C)_0$, 
  $$(\gamma^{\mins}_C)_0(v)\leqslant_{\Tau_1C}(\gamma^{\pls}_C)_0(v).$$
\end{lem}
\begin{proof}
  It suffices to consider the case $C=\BOX[1]$ by $\gamma^{\mins},\gamma^{\pls},\Tau_1$ monoidal and naturality.  
  In that case, the identity follows by exhaustive verification on the three possibilities for $v$.  
\end{proof}

We now introduce a kind of mapping complex $\langle B,C\rangle$ which will combinatorially model the space of all uniform maps between pseudometric spaces modelled by $B,C$ [Theorem \ref{thm:enriched.equivalence}]. 
The construction $\langle\star,-\rangle$ will turn out to function as a cubical analogue of Kan's $\mathrm{Ex}^\infty$ construction \cite{kan1957css} [Corollary \ref{cor:classical.we}]; the only formal difference between $\mathrm{Ex}^\infty$ and $\langle\star,-\rangle$ is that $\langle\star,-\rangle$ applies sharp replacement before applying a transfinite composite of $\ex$.  

\begin{defn}
  \label{defn:lipschitz.mapping.complex}
  Let $\langle B,C\rangle$ denote the cubical set
  $$\langle B,C\rangle=\colim\left( (\ex\,C^{\sharp})^B\xra{(\adjointfirstvertexmap_{C^{\sharp}})^B}(\ex^2C^{\sharp} )^B\xra{(\adjointfirstvertexmap_{\ex\,C^{\sharp}})^B}(\ex^3C^{\sharp})^B\cdots\right)$$
  natural in cubical sets $B,C$ with $B$ connected.  
  Define $\langle \amalg_iB_i,C\rangle=\prod_i\langle B_i,C\rangle$ natural in coproducts $\amalg_iB_i$ in $\CUBICAL\SETS$ and cubical sets $C$.  
\end{defn}

The composite of inclusion $C\ira C^{\sharp}$ with $\hat\gamma^{\mins k}_C$ induces natural inclusions
$$\CUBICAL\SETS(B,C)\ira\langle B,C\rangle.$$

The following lemmas collectively assert that, in some sense, quadruple cubical subdivisions locally factor through representables.
This factorizability, the analogue of the factorizability of double barycentric simplicial subdivision through polyhedral complexes \cite[Lemma 4.4, Proposition 4.5]{jardine2004simplicial}, allows homotopies to be constructed in terms of the natural linear structure of geometric realizations of representables.  
Let $\Star_C(v)$ denote the \textit{closed star} of a vertex $v\in C_{0(0)}$ in a cubical set $C$, the subpresheaf of $C$ consisting of all images $A\subset C$ of representables in $C$ having vertex $v$.
Justifications are given at the end of the section.

\begin{lem}
  \label{lem:collapse.star}
  For all cubical sets $C$ and $v\in\sd^4C_{0}$, 
  $$(\firstvertexmap\lastvertexmap)^2_C(\Star_{\sd^2C})(v)\subset\support_{\sd^2}(v_*,C).$$
\end{lem}

\begin{lem}
  \label{lem:star.flower}
  Fix cubical set $C$ and image $A\subset\sd^2C$ of a representable.
  There exist:
  \begin{enumerate}
    \item unique minimal $B\subset C$ with $A\cap\sd^2B\neq\varnothing$; and
    \item unique retraction $\rho:A\ra A\cap\sd^2B$.
  \end{enumerate}
  Moreover, $A\cap\sd^2B$ is representable and $(\firstvertexmap\lastvertexmap)_C(A\ira\sd^2C)=(A\cap\sd^2B\ira\sd^2C)\rho$.
\end{lem}

The observations above were essentially made elsewhere \cite[Lemmas 6.11, 6.12]{krishnan2015cubical} but for a couple of inconsequential differences.
Firstly, the cubical sets considered elsewhere \cite{krishnan2015cubical} are presheaves over the minimal variant of $\BOX$, which can be regarded as only certain colimits of representable cubical sets in the sense of the current paper.  
However, the lemmas above follow from the case where all cubical sets are representable by naturality.  
Secondly, the particular subdivisions defined elsewhere \cite{krishnan2015cubical} yield different edge orientations.
However, the statements above are really statements about the cellular structures of topological realizations and all cubical subdivisions under consideration are the same at the level of CW realizations.    

\subsection{Pseudometric}\label{subsec:pseudometric.spaces}
An \textit{extended} \textit{pseudometric} on a set $X$ is a function 
$$\rho:X\times X\ra[0,\infty]$$
satisfying the following conditions:
\begin{enumerate}
  \item $\rho(x,x)=0$ for all $x\in X$
  \item $\rho(x,y)=\rho(y,x)$ for all $x,y\in X$
  \item $\rho(x,z)\leqslant\rho(x,y)+\rho(y,z)$ for all $x,y,z\in X$
\end{enumerate}
An extended pseudometric $\rho$ is an \textit{extended metric} if $\rho$ nevers takes the value $0$ on pairs of distinct points and a \textit{metric} if $\rho$ never takes the values $0,\infty$ on pairs of distinct points.  
Take a \textit{(pseudo)metric space} to mean a set $X$ implicitly equipped with an (extended pseudo)metric on it.
We say that a pair $x,y$ in a pseudometric space $X$ are \textit{$\delta$-close} if the extended pseudometric on $X$ takes $(x,y)$ to a number no greater than $\delta$.  
The \textit{diameter} of a pseudometric space $X$ is the supremum of the image of the extended pseudometric.  
Riemannian manifolds will be regarded as metric spaces whose metrics are the Riemannian distance functions.  
Take an \textit{$\ell_p$-simplicial complex} to mean a simplicial complex equipped with an extended metric as defined in the following example.  

\begin{eg}
  \label{eg:metric.simplex}
  An \textit{$\ell_p$-simplex} is the topological simplex
  $$\nabla_{\ell_p}[n]=\{(t_0,t_1,\cdots,t_n)\in\I^{n+1}\;|\;t_0+\cdots+t_n=1\}$$
  with metric inherited from the metric on $\R^{n+1}$ equipped with its $p$-norm.  
  A $\ell_p$-simplicial complex is a simplicial complex $X$ regarded as a pseudometric space so that two points $x,y\in X$ are $\delta$-close if and only if $\sum_{i=1}^n\lambda_i<\delta$ for some sequence $x=x_0,x_2,\cdots,x_n=y$ of points in $X$ such that $x_{i-1},x_{i}$ are $\lambda_i$-close in a common simplex, naturally identified with a geometric $\ell_p$-simplex under a linear homeomorphism, for each $1\leqslant i\leqslant n$.
\end{eg}

A \textit{$\lambda$-(bi)-Lipschitz} function is a function $f:X\ra Y$ of pseudometric spaces such that for each $\delta>0$, $f(x),f(y)$ are $\lambda\delta$-close if (and only  if) $x,y$ are $\delta$-close. 
A (bi-)Lipschitz function is a $\lambda$-(bi-)Lipschitz function for some $\lambda>0$.  

\subsubsection{Constructions}
\textit{Pseudometric (co)limits} will simply refer to (co)limits in the category of pseudometric spaces and $1$-Lipschitz maps.
Pseudometric (co)limits are (co)limits of underlying sets equipped with certain universal extended pseudometrics constructed; we refer the reader elsewhere for details \cite{goubault2019directed}. 
For example, a product of pseudometric spaces is a Cartesian product of underlying sets together with the extended pseudometric defined by taking the suprema of coordinate-wise application of extended pseudometrics.

\begin{eg}
  In the product metic space $\I^n$, the distance between $x,y\in\I^n$ is
  $$\max\left(|x_1-y_1|,\cdots,|x_n-y_n|\right).$$
\end{eg}

Let $\METRIC$ denote the category of pseudometric spaces whose extended pseudometrics are extended \textit{metrics} and $1$-Lipschitz maps.
An $\METRIC$-isomorphism is exactly a surjective $1$-bi-Lipschitz map between $\METRIC$-objects.
All $\METRIC$-coproducts and $\METRIC$-limits are just special cases of pseudometric coproducts and equalizers.
The category $\METRIC$ is a reflective subcategory of the category of all pseudometric spaces and $1$-Lipschitz maps; the reflector naturally quotients a pseudometric space $X$ by the smallest equivalence relation so that the extended pseudometric in the quotient pseudometric space is an extended metric.  
Thus $\METRIC$-colimits can be constructed by taking pseudometric colimits and applying the reflector onto $\METRIC$.  
The following lemma therefore holds because suprema commutes with infima in $[0,\infty]$.

\begin{lem}
  \label{lem:metric.product.cocontinuity}
  For each $\METRIC$-object $X$, the functor
  $$X\times_{\METRIC}-:\METRIC\ra\METRIC$$
  preserves $\METRIC$-colimits that are already pseudometric colimits.
\end{lem}

\subsection{Uniform}\label{subsec:uniform.spaces}
A map $f:X\ra Y$ of pseudometric spaces is \textit{uniform} if for all $\epsilon>0$, there exists $\delta>0$ such that for all $x_1,x_2\in X$, $f(x_1),f(x_2)$ are $\epsilon$-close whenever $x_1,x_2$ are $\delta$-close.
A function between pseudometric spaces is uniform if it is Lipschitz, but need not be Lipschitz even if it is uniform.
Pseudometrics can be abstracted up to uniform equivalence as follows.
In fact, it is possible to define two different metrics $\rho_1,\rho_2$ on the same set $X$ such that the identity function $\id_X$ defines uniform maps $X_1\ra X_2$ and $X_2\ra X_1$, where $X_i$ is $X$ equipped with $\rho_i$.  
A \textit{uniformity} on a set $X$ is a collection $\mathscr{E}$ of reflexive relations $X\relation X$, called \textit{entourages}, such that the following axioms hold:
\begin{enumerate}
  \item $E^{-1}\in\mathscr{E}$ whenever $E\in\mathscr{E}$
  \item for each $E\in\mathscr{E}$, there exists $E^{\half}\in\mathscr{E}$ with $\graph{(E^{\half}\circ E^{\half})}\subset\graph{(E)}$.  
  \item for all $E_1,E_2\in\mathscr{E}$, there exists $E\in\mathscr{E}$ with $\graph{(E)}=\graph{(E_1)}\cap\graph{(E_2)}$.  
  \item A reflexive relation $R:X\relation X$ lies in $\mathscr{E}$ whenever there exists an entourage $E\in\mathscr{E}$ with $\graph{(R)}\subset\graph{(E)}$.  
\end{enumerate}

The first axiom abstracts the symmetry axiom of a metric.  
The second axiom abstracts the triangle inequality of a metric.  
The latter two axioms guarantee that the entourages form a neighborhood basis with respect to a suitable topology. 
Pseudometric spaces $X$ will be regarded as uniform spaces equipped with the following \textit{uniformity associated to the pseudometric} on $X$ in the following example.  
A uniform space is \textit{separated} if the intersection of the graphs of all its entourages is the diagonal in $X^2$.  
A \textit{uniform space} is a set with a uniformity on it.  
Every pseudometric space $X$ with extended pseudometric $\rho$ has an \textit{underlying uniform space} having the same underlying set as $X$ and whose uniformity is the smallest uniformity whose entourages include all relations of the form $E_\epsilon$ with
$$\graph{(E_\epsilon)}=\{(x,y)\in X^2\;|\;\rho(x,y)<\epsilon\}.$$

Pseudometric spaces will often be implicitly taken to mean their underlying uniform spaces.  
In this manner, we regard Riemannian manifolds are examples of uniform spaces.  
A \textit{uniform map} will be taken to mean a function between uniform spaces that pulls back entourages to entourages.

\begin{eg}
  Every Lipschitz map is uniform.
\end{eg}

\begin{eg}
  Every surjective bi-Lipschitz map is an isomorphism in $\UNIFORM$.  
\end{eg}

The \textit{uniform topology} of a uniform space $X$ is the completely regular topology on $X$ with respect to which a neighborhood basis for each $x\in X$ is the collection of all sets of the form $E[x]$ for entourages $E$ of $X$.
The \textit{underlying space} of a uniform space $X$ will be taken to mean the underlying set of $X$ equipped with the uniform topology of $X$. 
A uniform space is separated if and only if its underlying space is Tychonoff.
Underlying spaces of uniform spaces are completely regular.
Every completely regular space $X$ admits a \textit{fine uniformity}, the unique maximal uniformity turning $X$ into a uniform space whose uniform topology coincide with the original topology on $X$.  
Every uniform map defines a map of underlying spaces.

\begin{thm:compact.uniform.spaces}
  Each compact Hausdorff space underlies exactly one uniform space.
\end{thm:compact.uniform.spaces}

In this manner, each compact Hausdorff space will be automatically regarded as a uniform space.

\subsubsection{Constructions}
\textit{Uniform (co)limits} will simply refer to (co)limits in the category of (uniform) spaces and (uniform) maps.
The underlying spaces of uniform colimits are colimits of underlying spaces \cite[Proposition \S II.8]{isbell1964uniform}.
The categories of (uniform) spaces and (uniform) maps, while complete and cocomplete, are not Cartesian closed.  
This defect can be remedied in a standard way (eg. \cite{steenrod1967convenient}), except that the role of compact Hausdorff spaces in the definition of compactly generated spaces is played by connected, locally connected metric spaces.  
Write $\TOP$ for the category of Tychonoff spaces which are colimits of compact, locally connected, second-countable spaces.
The coreflective hull, in an epi-reflective subcategory $\mathscr{R}$ of spaces and maps, of a class of compact Hausdorff spaces closed under finite products forms a Cartesian closed coreflective subcategory \cite[Theorem 3.3]{wyler1973convenient} of $\mathscr{R}$.
The following theorem is a consequence.

\begin{prop}
  The category $\TOP$ satisfies the following:
  \begin{enumerate}
	\item $\TOP$ is Cartesian closed
	\item $\TOP$ is a coreflective subcategory of the category of Tychonoff spaces and continuous maps between them.   
  \end{enumerate}
\end{prop}

Let $\UNIFORM$ denote the category of separated uniform colimits of connected, locally connected metric spaces.
The category $\UNIFORM$ contains examples of interest in nature.

\begin{eg}
  Connected Riemannian manifolds are $\UNIFORM$-objects.
\end{eg}

The category of separated uniform spaces and uniform maps is a reflective subcategory of the category of all uniform spaces and uniform maps.  
Inside the former reflective subcategory, the full subcategory whose objects are uniform colimits of connected metric spaces is a Cartesian closed coreflective subcategory \cite[Corollary 1]{rice1983cartesian}.
Connected, locally connected metric spaces are closed under finite uniform products.  
The following proposition therefore follows.  

\begin{prop}
  The category $\UNIFORM$ satisfies the following:
  \begin{enumerate}
	\item $\UNIFORM$ is Cartesian closed
	\item $\UNIFORM$ is a coreflective subcategory of the category of separated uniform spaces and uniform maps
  \end{enumerate}
\end{prop}

Thus $\UNIFORM$-colimits are constructed by taking uniform colimits and applying the reflector to the category of separated uniform spaces and uniform maps.  
And $\UNIFORM$-limits are constructed by taking uniform limits and applying the coreflector to $\UNIFORM$.  
Write $(-)^X$ for the right adjoint to the endofunctor $X\times-$ on $\UNIFORM$, natural in $\UNIFORM$-objects $X$.
We henceforth regard $\UNIFORM$ as $\TOP$-enriched by naturally equipping each hom-set $\UNIFORM(X,Y)$ with the uniform topology of $Y^X$.

\begin{prop}
  \label{prop:fine.uniform.spaces}
  A $\TOP$-object with its fine uniformity is a $\UNIFORM$-object.
\end{prop}
\begin{proof}
  Consider a $\TOP$-object $X$ equipped with its fine uniformity. 
  Every compact, locally connected, secound-countable space is a a colimit of compact, connected, locally connected, metrizable spaces.  
  Then $X$ is the colimit of compact, connected, locally connected metrizable spaces with their fine uniformities because the forgetful functor from separated uniform spaces and uniform maps to Tychonoff spaces and maps has a left adjoint naturally equipping each space with its fine uniformity.
  Each compact, connected, locally connected, metrizable space with its fine uniformity is the underlying uniform space of a connected, locally connected metric space by the uniqueness of a compatible uniformity on a compact Hausdorff space.  
  Threfore $X$ is a $\UNIFORM$-object.
\end{proof}

\begin{cor}
  The $\TOP$-enriched category $\UNIFORM$ is tensored and cotensored over $\TOP$.
\end{cor}
\begin{proof}
  Let $S$ denote a $\TOP$-object.
  We can define a functor
  $$F:\TOP\ra\UNIFORM$$
  naturally equipping $S$ with its fine uniformity [Proposition \ref{prop:fine.uniform.spaces}].
  Then
  $$\TOP(S,\UNIFORM(X,Y))\cong\UNIFORM(FS,Y^X)\cong\UNIFORM(FS\times_{\UNIFORM}X,Y)=\UNIFORM(X,Y^{FS})$$
  naturally in $S$ and $\UNIFORM$-objects $X,Y$.  
\end{proof}

It therefore follows that the $\TOP$-enriched category $\UNIFORM$ is tensored and cotensored over $\TOP$; the tensor is a restriction of the Cartesian monoidal product on $\UNIFORM$ after equipping a $\TOP$-space with its fine uniformity.  

\subsection{Comparisons}\label{subsec:comparisons}
We will construct functors in the commutative diagram
\begin{equation}
	\label{eqn:comparisons}
	\begin{tikzcd}
		\METRIC
		  \ar[rr]
		& 
		& \UNIFORM\ar[rr]
		& 
		& \TOP
		\\
		  \BOX
		  \ar{rrrr}[below]{\BOX[-]}\ar{u}[left]{\boxobj{n}\mapsto\I^n}
		& 
		& 
		&
		& \CUBICAL\SETS
	  	    \ar{ullll}[description]{\geometricrealization{\;-\;}}
			\ar{ull}[description]{\uniformrealization{\;-\;}}
		    \ar{u}[right]{|-|}
	\end{tikzcd},
\end{equation}

\subsubsection{Forgetful functors}
Let the unlabelled solid arrows in (\ref{eqn:comparisons}) denote \textit{forgetful functors} defined as follows.
The left forgetful functor in (\ref{eqn:comparisons}) naturally sends each $\METRIC$-object $X$ to the coreflection in $\UNIFORM$ of the underlying uniform space of $X$.
The right forgetful functor in (\ref{eqn:comparisons}) naturally sends each $\UNIFORM$-object $X$ to the coreflection in $\TOP$ of the underlying space of $X$.  

\begin{lem}
  \label{lem:pseudometric.to.uniform}
  The forgetful functor $\METRIC\ra\UNIFORM$ preserves coproducts and finite products.  
\end{lem}
\begin{proof}
  Every uniform map from a connected, locally connected metric space to an $\METRIC$-coproduct factors through a summand.
  The desired coproduct-preservation follows.

  The underlying uniform space of a finite $\METRIC$-product, a special case of a finite pseudometric product, is the finite uniform product of underlying separated uniform spaces.  
  The coreflector from the category of separated uniform spaces and uniform maps to $\UNIFORM$, a right adjoint, preserves finite products.  
  The desired finite product-preservation follows.   
\end{proof}

Henceforth redefine a \textit{space} to mean a $\TOP$-object and a \textit{uniform space} to mean a $\UNIFORM$-object.

\subsubsection{Realizations}
Let the left vertical arrow in (\ref{eqn:comparisons}) denote the monoidal functor naturally sending each $\BOX$-morphism $\phi:\boxobj{m}\ra\boxobj{n}$ to the $1$-Lipschitz map $\I^m\ra\I^n$ whose restriction to the convex closure in $\I^m$ of each chain in the poset $\boxobj{m}\subset\I^m$ is linear.
Let $\geometricrealization{-},\uniformrealization{-},|-|$ in (\ref{eqn:comparisons}) respectively denote \textit{geometric}, \textit{uniform}, and \textit{topological} realization functors making the diagram commute with the former cocontinuous.

\begin{eg}
  The metric space $\R$ is isomorphic in $\METRIC$ to the geometric realization
  $$\geometricrealization{\cdots\ra\bullet\ra\bullet\ra\cdots}.$$
\end{eg}

\begin{lem}
  \label{lem:bounded.nerve}
  For a cubical set $C$ with $\BOX[-]/S$ filtered, $\geometricrealization{C}$ has diameter $1$.
\end{lem}
\begin{proof}
  Each point $z\in|C|$ lies in the image of a cubical function 
  $$|\mu_z|:|\BOX[n_z]|\ra|C|.$$
  
  Consider $x,y\in|C|$.  
  There exists $(\BOX[-]/C)$ object $\mu_{xy}:\BOX[n_{xy}]\ra C$ to which there exist $(\BOX[-]/C)$-morphisms from $\mu_x,\mu_y$ by $\BOX[-]/C$ filtered.
  Therefore $x,y$ lie in the image of the $1$-Lipschitz map $|\mu_{xy}|$ whose domain $|\BOX[n_{xy}]|$ has diameter $1$.
\end{proof}

\begin{eg}
	For each monoid $M$, $\geometricrealization{\cnerve M}$ has diameter $1$.
\end{eg}

\begin{eg}
	For each connected Kan complex $C$, $\geometricrealization{C}$ has diameter $1$.
\end{eg}

Fix a cubical set $C$.  
We henceforth identify the underlying sets of all realizations of $C$.
The space $|C|$ will be regarded as a CW complex whose open $n$-cells are images of $|\BOX\boxobj{n}|-|\partial\BOX\boxobj{n}|$ under topological realizations $|\theta_*:\BOX\boxobj{n}\ra C|$ for non-degenerate $n$-cubes $\theta$ of $n$.  
In particular, each open $0$-cell of $|C|$ can be naturally identified with a vertex in $C$.  
The \textit{open star} of a vertex $v$ in $|C|$ is the union of the open $n$-cells in $|C|$ containing $v$.

\begin{eg}
  The closed star of a $0$-cell $|v|$ in $|C|$ is $|\Star_C(v)|$. 
\end{eg}

Fix a cubical set $C$.  
We give some explicit descriptions of the uniformity structure of $\uniformrealization{C}$ as follows.
Unravelling the construction of colimits in $\METRIC$, the following are equivalent for $x,y\in|C|$ and $0<\epsilon<\infty$:
\begin{enumerate}
	\item $x,y$ are $\epsilon$-close
	\item there exist $n\geqslant 1$, $x=x_0,x_1,\ldots,x_n=y\in|C|$ and $\epsilon_1,\epsilon_2,\ldots,\epsilon_n\geqslant 0$ such that $\sum_i\epsilon_i\leqslant\epsilon$ and for each $1\leqslant i\leqslant n$, $x_{i-1},x_i$ are the images of $\epsilon_i$-close points under a map of the form $\geometricrealization{\BOX\boxobj{n_i}\ra C}$.
\end{enumerate}
It follows that for connected cubical sets $C$, $\geometricrealization{C}$ is a connected, locally connected metric space whose underlying uniform space is $\uniformrealization{C}$.

\begin{lem}
  \label{lem:star.proximity}
  Fix a cubical set $C$.  
  The following hold for $x,y\in\geometricrealization{C}$.
  \begin{enumerate}
    \item If $x,y$ lie in the open star of a common vertex, then $x,y$ are $2$-close.
	\item If $x,y$ are separated by distance less than $1$, then $x,y$ lie in the open star of a common vertex.
  \end{enumerate}
\end{lem}
\begin{proof}
  It suffices to consider the case $C$ connected.
  There exists a finite sequence $x=x_1,\ldots,x_n=y\in\geometricrealization{S}$, for a minimal choice of $n\geqslant 2$, such that $x_i,x_{i+1}$ both lie in the same closed cell for each $1\leqslant i<n$.\\[.05cm]

  \textit{proof of (1):}
  Suppose $x,y$ lie in the open star of a common vertex.
  Then there exists a vertex $v$ in $C$ such that $x,v$ and $v,y$ are each $1$-close.  
  Hence $x,y$ are $2$-close by the triangle inequality.\\[.05cm]
  
  \textit{proof of (2):}
  Suppose $x,y$ do not lie in the open star of a common vertex.
  Then $n>3$.  
  Then $x_2,x_3$ lie in a common closed cell but cannot lie in the open star of a common vertex by minimality of $n$.  
  Therefore $x_2,x_3$ are separated by distance $1$.
  Hence $x,y$ are separated by at least distance $1$ by the triangle inequality.
  Hence (2).  
\end{proof}

Write $\varphi_S$ for the natural and monoidal homeomorphism $\varphi_C:|\sd\,C|\cong|C|$, linear on each $1$-cell of $\varphi_{\BOX[1]}$ and sending the unique $0$-cell in $|\sd\,\BOX[1]|$ that is a cut-point to $\half$.
Then $\varphi_S$ defines a $\half$-bi-Lipschitz homeomorphism $\geometricrealization{\sd\,C}\ra\geometricrealization{C}$ and hence a $\UNIFORM$-isomorphism
$$\varphi_C:\uniformrealization{\sd\,C}\cong\uniformrealization{C}.$$

\begin{lem}
  \label{lem:uniform.characterization}
  Fix cubical sets $A,B$ with $A$ connected. 
  Consider a function
  $$f:\geometricrealization{A}\ra\geometricrealization{B}.$$
  The following are equivalent. 
  \begin{enumerate}
	\item $f$ defines a uniform map $\uniformrealization{A}\ra\uniformrealization{B}$
	\item For each $b>0$, there exists $a>0$ such that $\varphi^{-b}_Bf\varphi^a_A:|\sd^aA|\ra|\sd^bB|$ maps each closed cell into the open star of a vertex.
  \end{enumerate}
\end{lem}
\begin{proof}
  Assume (1). 
  Then for each $b>0$, there exists $a>0$ such that $f$ maps $2^{-a}$-close points in $|A|$ to $2^{-b}$-close points in $|B|$, or equivalently $\varphi^{-b}_Bf\varphi^a_A$ maps $1$-close points in $|\sd^aA|$ to $1$-close points in $|\sd^bB|$, and in particular closed cells in $|\sd^aA|$ into open stars in $|\sd^bB|$ [Lemma \ref{lem:star.proximity}].  
  Hence (2).  

  Assume (2).  
  Consider $\epsilon>0$.  
  Let $b=\lceil\log_2\epsilon\rceil+1$.  
  There exists $a>0$ such that $\varphi^{-b}_Bf\varphi^a_A$ maps each closed cell into the open star of a vertex.
  Then $\varphi^{-b}_Bf\varphi^a_A$ maps $1$-close points to $2$-close points [Lemma \ref{lem:star.proximity}], or equivalently $f$ maps $2^{-a}$-close points to $2^{-\lceil\log_2\epsilon\rceil}$-close, and hence $\epsilon$-close, points.
  Hence (1).
\end{proof}

\begin{lem}
  \label{lem:dilation.formula}
  For all cubical sets $C$ and $v\in(\sd\,C)_0$, 
  $$(\gamma^{\mins}_C)_0(v)=(\support_{|\sd-|}(|v_*|,C))_0(0,\cdots,0)\quad (\gamma^{\pls}_C)_0(v)=(\support_{|\sd-|}(|v_*|,C))_0(1,\cdots,1).$$
\end{lem}
\begin{proof}
  It suffices to show the left equality, the right equality following similarly.  
  It also suffices to consider the case $C=\BOX[1]$ by $\gamma^{\pm},|-|$ monoidal and naturality.  
  In that case, the identities follow from exhaustive verification on the three possibilities for $v$.  
\end{proof}

A proof of the following lemma uses intermediate simplicial constructions and is therefore deferred to the end of \S\ref{sec:simplicial}.  

\begin{lem}
  \label{lem:geometric.sharp.retraction}
  The $1$-Lipschitz map $\geometricrealization{C\ira C^\sharp}$ admits a $1$-Lipschitz retraction
  $$\geometricrealization{C^\sharp}\ra\geometricrealization{C}$$
  natural in cubical sets $C$.  
\end{lem}

\begin{prop}
  \label{prop:geometric.realization}
  Geometric realization is monoidal and preserves colimits.
\end{prop}
\begin{proof}
  Colimit preservation is by construction.
  Define $\rho_{AB}$ by the commutative diagram
  \begin{equation*}
	  \begin{tikzcd}
		  {\geometricrealization{A\otimes B}}\ar{r}[above]{\rho_{AB}}\ar{d}[left]{\geometricrealization{A\otimes B\ira A\times B}}
		  & {\geometricrealization{A}\times\geometricrealization{B}}\ar[d,equals]\\
		  {\geometricrealization{A\times B}}\ar[r]
		  & {\geometricrealization{A}\times\geometricrealization{B}}
	  \end{tikzcd}
  \end{equation*}
  natural in cubical sets $A,B$, where the bottom horizontal arrow is the canonoical map.  
  Under the natural identification $|\BOX\boxobj{n}|=\I^n$, $\rho_{\BOX\boxobj{p}\BOX\boxobj{q}}$ is the isomorphism $\I^p\times\I^q\ra\I^{p+q}$ in $\METRIC$.
  Hence $\rho_{A,B}$ is an isomorphism for general cubical sets $A,B$ because $\METRIC$-products commute with $\METRIC$-colimits [Lemma \ref{lem:metric.product.cocontinuity}]. 
\end{proof}

\begin{cor}
  \label{cor:uniform.realization}
  Uniform realization is monoidal and preserves coproducts.
\end{cor}
\begin{proof}
  Uniform realization is a composite of monoidal, coproduct-preserving functors [Lemma \ref{lem:pseudometric.to.uniform} and Proposition \ref{prop:geometric.realization}].
\end{proof}

\section{Homotopy}\label{sec:homotopy}
Homotopy theories for cubical sets and uniform spaces, both classical and geometric, are formalized in \S\ref{subsec:cubical.homotopy} and \S\ref{subsec:uniform.homotopy}.
The Lipschitz homotopy theory of cubical sets and the uniform homotopy theory of uniform spaces are compared in \S\ref{subsec:homotopical.comparisons}.  
In each of the homotopy theories, the given category $\mathscr{X}$ is equipped with \textit{cylinder objects} $\odot$, functors $\odot:\mathscr{X}\times\BOX_1\ra\mathscr{X}$ for which $-\odot[0]\cong\id_{\mathscr{X}}$; a \textit{$\odot$-homotopy} between parallel morphisms $\zeta_\pm:x\ra y$ can defined as some sort of morphism $\eta:x\odot[1]\ra y$ for which $\eta(x\odot\delta_{\pm})=\zeta_{\pm}$ for each cylinder object $\odot$.  
We will repeatedly use the following two formal facts about such homotopies without additional comment.
The first is that a $\odot$-homotopy between both possible projections of the form $o^2\ra o$ naturally induces a $\odot$-homotopy between any pair of parallel morphisms to $o$.  
The second is that for each lax Cartesian monoidal functor $F$ between two such categories equipped with respective cylinder objects $\odot_1,\odot_2$ for which there exists a natural transformation $(F-)\odot_2-\ra F(-\odot_1-)$ whose component $Fx\odot_2[0]\ra F(x\odot_1[0])=\id_x$ for all objects $x$, there exists a $\odot_2$-homotopy between projections of the form $(Fy)^2\ra Fy$ natural in $\odot_1$-homotopies between projections of the form $y^2\ra y$.  

\subsection{Cubical}\label{subsec:cubical.homotopy}
We define \textit{Lipschitz weak equivalences} and  \textit{classical weak equivalences} as progressive generalizations of \textit{cubical homotopy equivalences}, defined in terms of \textit{cubical homotopies} as follows.  
A \textit{cubical homotopy} $\alpha\sim\beta$ between cubical functions $\alpha,\beta:B\ra C$ is a dotted cubical function making the following diagram for some $k\gg 0$:
\begin{equation*}
  \xymatrix{
    **[l]B\amalg B\ar[r]\ar[d]_{B\otimes(\partial\BOX[1]\ira\BOX[1])} & **[r]C\\
	**[l]B\otimes\sd^k\BOX[1]\ar@{.>}[ur]
	}
\end{equation*}

For example, there always exist cubical homotopies between a pair of parallel cubical functions to $\cnerve\boxobj{n}$ by the following lemma.

\begin{lem}
  \label{lem:sharp.local.convexity}
  There exists a cubical homotopy between both projections of the form
  $$(\cnerve\boxobj{n})^2\ra\cnerve\boxobj{n}$$
  natural in $\BOX$-objects $\boxobj{n}$. 
\end{lem}
\begin{proof}
  There exist monotone functions $\eta_{1;n},\eta_{2;n}$ of the form 
  $$\eta_{1;n},\eta_{2;n}:(\boxobj{n})^2\times[1]\ra\boxobj{n}$$
  defined by the following rules for $x,y\in\boxobj{n}$:
  \begin{align*}
	  \eta_{1;n}(x,y,0)&=x\\
	  \eta_{2;n}(x,y,0)&=y\\
	  \eta_{1;n}(x,y,1)&=\eta_{2;n}(x,y,1)=x\vee_{\boxobj{n}}y.
  \end{align*}

  The functions $\eta_{1;n},\eta_{2;n}$ are natural in $\BOX$-objects $\boxobj{n}$ because $\BOX$-morphisms are sup-semilattice homomorphisms.
  The functor $\cnerve$, a right adjoint, preserves products.
  It is therefore possible to define cubical homotopies $\eta'_{1;n},\eta'_{2;n}$ by the commutative diagram
  \begin{equation*}
	 \begin{tikzcd}
	   (\cnerve\boxobj{n})^{\otimes 2}\otimes\BOX[1]
	   \ar{rrrr}[above]{\eta''_{i;n}}
	    \ar{d}[description]{(\cnerve\boxobj{n})^{\otimes 2}\otimes\BOX[1]\ira(\cnerve\boxobj{n})^{\times 2}\times\BOX[1]} 
	  & & & & \cnerve\boxobj{n}
	  \\
	    (\cnerve\boxobj{n})^{\times 2}\times\BOX[1]
		\ar{rrrr}[below]{(\cnerve\boxobj{n})^{\times 2}\times(\BOX[1]\ira\cnerve[1])}
		& & & & \cnerve(\boxobj{n})^{\times 2}\times [1])\ar{u}[right]{\cnerve\eta_{i;n}},
     \end{tikzcd}
  \end{equation*}
  natural in $\BOX$-objects $\boxobj{n}$, between both subprojections of the form $(\cnerve\boxobj{n})^{\otimes 2}\ra(\cnerve\boxobj{n})$ and a common cubical function. 
  Thus a cubical function of the form $(\cnerve\boxobj{n})^{\otimes 2}\otimes\sd\BOX[1]\ra\cnerve\boxobj{n}$ defines the desired cubical homotopy.
\end{proof}

A \textit{cubical homotopy inverse} to a cubical function $\alpha:A\ra B$ is a cubical function $\beta:B\ra A$ for which there exist cubical homotopies $\beta\alpha\sim\id_{A}$ and $\alpha\beta\sim\id_{B}$.
A \textit{cubical homotopy equivalence} is a cubical function $\psi:B\ra C$ admitting a cubical homotopy inverse.
A cubical function $\psi:B\ra C$ is a cubical homotopy equivalence precisely if the following is a bijection for each cubical set $A$:
$$\pi_0\psi^A:\pi_0B^A\cong\pi_0C^A.$$

Examples are inclusions of the form $C^\sharp\ira C^{\sharp\sharp}$ by the following lemma.

\begin{lem}
  \label{lem:sharp.idempotency}
  There exists a retraction $\rho_{C}$ to the inclusion 
  $$C^\sharp\ira C^{\sharp\sharp}$$
  and cubical homotopy $(C^\sharp\ira C^{\sharp\sharp})\rho_{C}\sim\id_{C^{\sharp\sharp}}$ natural in cubical sets $C$.
\end{lem}
\begin{proof}
  It suffices to consider the case $C$ of the form $\BOX\boxobj{n}$ by naturality and cocontinuity of $(-)^\sharp$.  
  In that case, $C^\sharp=\cnerve\boxobj{n}$.
  Inclusion $\iota_n:\cnerve\boxobj{n}\ira(\cnerve\boxobj{n})^\sharp$ admits a natural retraction $\rho_n$ [Lemma \ref{lem:sharp.sharp}] natural in $\BOX$-objects $\boxobj{n}$.
  There exists a cubical homotopy, natural in $\BOX$-objects $\boxobj{n}$, between projections of the form $((\cnerve\boxobj{n})^{\sharp})^2\ra(\cnerve\boxobj{n})^{\sharp}$ [Lemma \ref{lem:sharp.local.convexity}].
  It therefore follows that there exists a cubical homotopy $\iota_n\rho_n\sim\id_{(\cnerve\boxobj{n})^\sharp}$ natural in $\BOX$-objects $\boxobj{n}$.
\end{proof}

\begin{lem}
  \label{lem:sharp.whitehead}
  For each cubical set $C$, the inclusion
  $$C_\sharp\ira(C_\sharp)^\sharp$$
  is a cubical homotopy equivalence.
\end{lem}
\begin{proof}
  Consider a $(\BOX[-]/C_\sharp)$-object $\sigma$ in the diagram
  \begin{equation*}
    \begin{tikzcd}
	  \BOX\boxobj{n}\ar{r}[above]{\sigma}\ar[d,hookrightarrow] & C_\sharp\ar[d]\\
	  \cnerve\boxobj{n}\ar{r}[below]{\hat\sigma}\ar[ur,dotted] & C
	\end{tikzcd}
  \end{equation*}
  Let the right vertical arrow above be the adjoint to the natural inclusion $C\ira C^\sharp$.
  There exists a cubical function $\hat\sigma$, natural in $\sigma$, making the entire solid diagram commute. 
  And every cubical function $\BOX\boxobj{m}\ra\cnerve\boxobj{n}$ uniquely extends to a cubical function $\cnerve\boxobj{m}\ra\cnerve\boxobj{n}$.  
  Therefore there exists a dotted cubical function natural in $\sigma$, making the entire diagram commute.  
  Thus $C_\sharp=\colim_{\cnerve\boxobj{n}\ra C_\sharp}\cnerve\boxobj{n}$.  
  Thus cubical homotopy equivalences of the form $\cnerve\boxobj{n}\ira(\cnerve\boxobj{n})^\sharp$, for which associated cubical homotopy inverses and cubical homotopies are all natural in $\BOX$-objects $\boxobj{n}$ [Lemma \ref{lem:sharp.idempotency}], induce the desired cubical homotopy equivalence.
\end{proof}

\subsubsection{Classical}
Take a cubical function $\psi:B\ra C$ to be a \textit{classical weak equivalence} if 
$$\pi_0\langle\star,\psi\rangle^{A}:\pi_0\langle\star,B\rangle^{A}\cong\pi_0\langle\star,C\rangle^{A}$$
for all cubical sets $A$. 
It will later turn out that classical weak equivalences in our sense are exactly the homotopy equivalences after applying topological realization [Corollary \ref{cor:classical.we}].
Therefore the classical weak equivalences are part of the following \textit{test model structure} on $\CUBICAL\SETS$.
\cite[Proposition 4.3]{maltsiniotis2009categorie}, \cite[Theorem 1.7]{cisinskiprefaisceaux}.  

\begin{thm}
  \label{thm:cubical.model.structure}
  There exists a model structure on $\CUBICAL\SETS$ in which \ldots
  \begin{enumerate}
	  \item \ldots the weak equivalences are the classical weak equivalences
	  \item \ldots the cofibrations are the monos
	  \item \ldots the fibrant objects are the Kan cubical sets
  \end{enumerate}
\end{thm}

It follows from the theorem that that there exists a localization 
$$\CUBICAL\SETS\ra h(\CUBICAL\SETS)$$ 
of $\CUBICAL\SETS$ by the classical weak equivalences.
Classical homotopy theory admits the following Abelianization.  
Let $\chains_*C$ be the integral cellular chain complex of the CW complex $|C|$ natural in cubical sets $C$.  
The \textit{Dold-Kan correspondence} is an adjoint categorical equivalence defined by the following theorem.  

\begin{thm:dold-kan}
   There exists an adjoint categorical equivalence
   $$\CUBICAL\ABGROUPS\simeq\CONNECTIVECHAINCOMPLEXES$$
   whose right adjoint naturally sends a connective chain complex $C$ to $\CONNECTIVECHAINCOMPLEXES\left(\chains_*\BOX[-],C\right)$.
\end{thm:dold-kan}

Write $\Z[-]$ for the left adjoint to the functor $\CUBICAL\ABGROUPS\ra\CUBICAL\SETS$ naturally sending a cubical Abelian group $C$ to its underlying cubical set, the composite of $C$ with the forgetful functor $\ABGROUPS\ra\SETS$.  

\begin{eg}
  Under the Dold-Kan correspondence, $\chains_*C$ corresponds to $\Z[C]$.
\end{eg}

\subsubsection{Lipschitz}
Lipschitz homotopy theory for cubical sets is based on the following refinement of a classical weak equivalence.  
Lipschitz weak equivalences, at least between connected cubical sets, will turn out to correspond exactly to Lipschitz homotopy equivalences [Corollary \ref{cor:cubical.we}].   

\begin{defn}
  \label{defn:lipschitz.we}
  A cubical function $\psi:B\ra C$ is a \textit{Lipschitz weak equivalence} if 
  $$\langle A,\psi\rangle:\langle A,B\rangle\ra\langle A,C\rangle$$
  is a classical weak equivalence of underlying cubical sets, for all cubical sets $A$.
\end{defn}

\begin{prop}
  \label{prop:strong.to.weak.lipschitz.equivalences}
  Every cubical homotopy equivalence is a Lipschitz weak equivalence.
\end{prop}
\begin{proof}
  Consider a cubical set $A$.  
  The functors $\ex$ and $(-)^A$ preserve cubical homotopy equivalences.  
  Therefore for each cubical homotopy equivalence $\psi$, $\langle A,\psi\rangle$ is cubical homotopy equivalence and in particular a classical weak equivalence.
\end{proof}

Proofs of the following observations, easiest to give after cubical approximation, are deferred to the end of \S\ref{subsec:homotopical.comparisons}.

\begin{prop}
  \label{prop:lipschitz.implies.classical}
  Every Lipschitz weak equivalence is a classical weak equivalence.
\end{prop}

\begin{prop}
  \label{prop:finite.lipschitz}
  The following are equivalent for a cubical function
  $$\psi:B\ra C$$
  between finite cubical sets $B$ and $C$.
  \begin{enumerate}
	  \item $\psi$ is a classical weak equivalence
	  \item $\psi$ is a Lipschitz weak equivalence
  \end{enumerate}
\end{prop}

\begin{prop}
  \label{prop:cofibrant.objects}
  The category $\CUBICAL\SETS$ is a category of cofibrant objects in which \ldots
  \begin{enumerate}
    \item \ldots the weak equivalences are the Lipschitz weak equivalences
	\item \ldots the cofibrations are the monos
	\item \ldots the fibrant objects are the Kan cubical sets
  \end{enumerate}
\end{prop}

It follows from the proposition that there exists a localization
$$\CUBICAL\SETS\ra h_\infty(\CUBICAL\SETS)$$
of $\CUBICAL\SETS$ by the Lipschitz weak equivalences.
Write $[-,-]$ for the hom-set functor $h_\infty(\CUBICAL\SETS(-,-))$.  
The hom-sets in the Lipschitz homotopy category $h_\infty(\CUBICAL\SETS)$ are just the connected components of the Lipschitz mapping complexes.

\begin{prop}
  \label{prop:lipschitz.weak.homotopy.classes}
  There exists a natural isomorphism
  $$[-,-]\cong\pi_0\langle -,-\rangle:\OP{(\CUBICAL\SETS)}\times\CUBICAL\SETS\ra\SETS.$$
\end{prop}

The fibrant objects in the test model structure for classical cubical homotopy and weaker structure for Lipschitz cubical homotopy coincide.  
On such fibrant objects, Lipschitz homotopy theory reduces to classical homotopy theory.  
In other words, cubical sets are too rigid a model of geometric objects to allow for fibrant replacement in Lipschitz homotopy theory.    

\subsection{Continuous}\label{subsec:uniform.homotopy}
We define homotopy theories on (uniform) spaces in this section.
For convenience, we adopt the following terminology.
Fix a $\TOP$-category $\homotopicalcat{1}$ tensored and cotensored over $\TOP$.  
An \textit{h-homotopy} $\alpha\sim\beta$ between $\homotopicalcat{1}$-morphisms $\alpha,\beta$ is a dotted $\homotopicalcat{1}$-morphism making the left of the followingn diagrams commute:
\begin{equation*}
  \xymatrix{
    **[l]x\otimes\{0,1\}=x\amalg x\ar[r]\ar[d]_{x\otimes(\{0,1\}\ira\I)} & **[r]y\\
	**[l]x\otimes\I\ar@{.>}[ur]
	}\quad
  \xymatrix{
    **[l]a\otimes\I\ar[rr]^{a\otimes(\I\ra\{0\})}\ar[d]_{i\otimes\I} & & **[r]A\ar[d]^{\alpha\iota}\\
	**[l]x\otimes\I\ar[rr]_{\eta} & & y
	}
\end{equation*}
  
An h-homotopy $\eta:\alpha\sim\beta$ in $\homotopicalcat{1}$ is \textit{relative} a $\homotopicalcat{1}$-morphism $\iota:a\ra x$ if the right diagram above commutes.
A pair of parallel $\homotopicalcat{1}$-morphisms $\alpha,\beta$ are \textit{h-homotopic} if there exists a homotopy between them.
Call a $\homotopicalcat{1}$-morphism $\zeta$ an \ldots
\begin{enumerate}
	\item \ldots \textit{h-equivalence} if $\pi_0\homotopicalcat{1}(-,\zeta)$ is object-wise bijective
	\item \ldots \textit{h-fibration} if $\homotopicalcat{1}(-,\zeta)$ object-wise has the right lifting property against $\{0\}\ira\I$
	\item \ldots \textit{h-cofibration} if $\zeta$ has the left lifting property against $o^{\{0\}\ira\I}$ for all $\homotopicalcat{1}$-objects $o$ 
\end{enumerate}

\subsubsection{Classical}
A \textit{classical homotopy} is an h-homotopy in $\TOP$.  
A \textit{classical homotopy equivalence} is an h-equivalence in $\TOP$.  
A \textit{Hurewicz (co)fibration} is an h-(co)fibration in $\TOP$.  
Recall that a map $f:X\ra Y$ of spaces is a \textit{classical weak equivalence} if, for all cubical sets $C$,
$$\pi_0\TOP(|C|,f):\pi_0\TOP(|C|,X)\cong\pi_0\TOP(|C|,Y).$$

The classical weak equivalences are probably part of a model structure on $\TOP$ equivalent to model structures on more ubiquitous variants of $\TOP$ in the literature (eg. \cite{quillen2006homotopical}); however, we will not need the existence of such a model structure in the current paper.  
Every space is classically weakly equivalent to the topological realization of a cubical set.  
And the localization of the full subcategory of $\TOP$ consisting of such topological realizations by the classical weak equivalences can be constructed so that the morphisms are classical homotopy classes of maps.  
It therefore follows that there exists a localization
$$\TOP\ra h\TOP$$
of $\TOP$ by the classical weak equivalences.
Call a uniform map a \textit{classical weak equivalence} if it defines a classical weak equivalence of underlying spaces.

\subsubsection{Uniform}
A \textit{uniform homotopy} is an h-homotopy in $\UNIFORM$.  
A \textit{uniform homotopy equivalence} is an h-equivalence in $\UNIFORM$. 
A \textit{uniform Hurewicz (co)fibration} is an h-(co)fibration in $\UNIFORM$.  
Uniform homotopy theory is more subtle than classical homotopy theory.  

\begin{eg}
  A homotopy through uniform maps need not be uniform.
  The rule
  $$h(x,t)=tx$$
  defines a homotopy $\R\times\I\ra\R$ through linear and hence uniform maps $\R\ra\R$ that itself is not uniform.  
  In fact, the uniform map $\R\ra\star$ from the usual metric space $\R$ is not a uniform homotopy equivalence.  
\end{eg}

Just as uniform homotopy equivalences refine classical homotopy equivalences, \textit{uniform weak equivalences} refine classical weak equivalences.

\begin{defn}
  \label{defn:uniform.we}
  A uniform map $f:X\ra Y$ is a \textit{uniform weak equivalence} if 
  $$\UNIFORM(|C|_\infty,f):\UNIFORM(|C|_\infty,X)\ra\UNIFORM(|C|_\infty,Y)$$
  is a classical weak equivalence of spaces, for all cubical sets $C$.  
\end{defn}

\begin{prop}
  A uniform weak equivalence is a classical weak equivalence.
\end{prop}
\begin{proof}
  For each uniform weak equivalence $f:X\ra Y$, 
  $$\UNIFORM(\star,f):\UNIFORM(\star,X)\ra\UNIFORM(\star,Y)$$
  is a classical weak equivalence that, up to homeomorphism, is the map of underlying spaces.
\end{proof}

\begin{prop}
  \label{prop:strong.to.weak.uniform.equivalences}
  A uniform homotopy equivalence is a uniform weak equivalence.
\end{prop}
\begin{proof}
  Consider a uniform homotopy equivalence $f:X\ra Y$ in $\UNIFORM$.
  Then $f$ represents an isomorphism in the quotient of $\UNIFORM$ by the congruence that equates path-connected maps in each hom-space by the Yoneda Lemma.    
  In other words, there exists a uniform map $g:Y\ra X$ together with uniform homotopies $gf\sim\id_X$ and $fg\sim\id_Y$.
  For all cubical sets $C$, $\UNIFORM(|C|_\infty,f)$ has homotopy inverse $\UNIFORM(|C|_\infty,g)$ and in particular is a classical weak equivalence.
\end{proof}

\begin{prop}
  \label{prop:fibrant.uniform.spaces}
  The category $\UNIFORM$ is a category of fibrant objects such that \ldots
  \begin{enumerate}
    \item \ldots the weak equivalences are the uniform weak equivalences
	\item \ldots the fibrations $f$ are the uniform Hurewicz fibrations
  \end{enumerate}
\end{prop}

The proof uses the fact that $\UNIFORM$ is a category of fibrant objects in which the weak equivalences are the uniform homotopy equivalences and the fibrations are the uniform Hurewicz fibrations [Example \ref{eg:h-fibrant}].

\begin{proof}
  For each cubical set $C$, $\UNIFORM(\uniformrealization{C},-)$ preserves pullbacks and sends uniform Hurewicz fibrations to classical Hurewicz fibrations.
  Therefore the class of uniform weak equivalences that are uniform Hurewicz fibrations is stable under pullbacks because classical Hurewicz fibrations of spaces are stable under pullback.  

  The rest of the axioms to check follow because $\UNIFORM$ is a category of fibrant objects whose weak equivalences are a special case of uniform weak equivalences [Proposition \ref{prop:strong.to.weak.uniform.equivalences}] and whose fibrations are the uniform Hurewicz fibrations.  
\end{proof}

It follows from the proposition that there exists a localization
$$\UNIFORM\ra h_\infty\UNIFORM$$
of $\UNIFORM$ by the uniform weak equivalences.
Write $[-,-]$ for the hom-set functor $h_\infty\UNIFORM(-,-)$.  
Call a uniform space \textit{uniformly cofibrant} if it is cofibrant in the above category of fibrant objects; in other words, a uniform space $X$ is uniformly cofibrant if every uniform map of the form $X\ra Y$ lifts along a uniform map to $Y$ that is at once a uniform weak equivalence and a uniform Hurewicz fibration.
Formally, a uniform map between uniformly cofibrant uniform spaces is a uniform weak equivalence if and only if it is a uniform homotopy equivalence.
The following sufficient condition for uniform cofibrancy generalizes a characterization of cofibrancy in the mixed model structure on spaces \cite{cole2006mixing}; a proof is formal and therefore omitted.

\begin{prop}
  \label{prop:cofibrant.uniform.spaces}
  Consider a cubical set $C$ and uniform homotopy equivalence
  $$X\simeq\uniformrealization{C}$$
  Then $X$ is uniformly cofibrant.
\end{prop}

\begin{cor}
  \label{cor:spaces}
  Spaces of finite CW type with fine uniformities are uniformly cofibrant.
\end{cor}

There exists homeomorphisms, one for each $n$, between a topological $n$-simplex and the topological realization of a cubical set, that collectively induce a homeomorphism between an $\ell_p$-simplicial complex and the geometric realization of a cubical set (eg. \cite{kapulkin2019co}.)
Each $\ell_p$-simplex of dimention $n$, a topological $n$-simplex with the $\ell_p$-metric on barycentric coordinates [Example \ref{eg:metric.simplex}], is bi-$\lambda_n$-Lipschitz equivalent to the geometric realization of a cubical set for some $\lambda_n>0$.  
Thus a finite dimensional simplicial complex $X$ is bi-Lipschitz equivalent to the geometric realization of a cubical set.
The corollary below follows.  

\begin{cor}
  \label{cor:cofibrant.manifolds}
  Finite dimensional $\ell_2$-simplicial complexes are uniformly cofibrant.
\end{cor}


\subsubsection{Lipschitz}
A \textit{Lipschitz homotopy} is a classical homotopy between Lipschitz maps $X\ra Y$ that defines a Lipschitz map $X\times\I\ra Y$.  
A \textit{Lipschitz homotopy equivalence} is a Lipschitz map $f:X\ra Y$ for which there exists another Lipschitz map $g:Y\ra X$ such that there exist Lipschitz homotopies $gf\sim\id_{X}$ and $fg\sim\id_{Y}$.   
Lipschitz homotopies are uniform homotopies.  
Lipschitz homotopy equivalences are uniform homotopy equivalences.

\subsection{Comparisons}\label{subsec:homotopical.comparisons}
The goal of this section is to investigate the degree to which uniform realization refines the following equivalence between classical homotopy categories of cubical sets and spaces.
Write $\sing_{\BOX}$ for the right adjoint to topological realization $|-|$.    
The following theorem follows from a Quillen equivalence between $\CUBICAL\SETS$ with its test model structure and a model category of traditionally defined spaces and maps between them \cite{quillen2006homotopical}.  

\begin{thm}
  \label{thm:equivalence}
  The adjunction $|-|\dashv\sing_\BOX$ passes to an adjoint categorical equivalence
  $$h(\CUBICAL\SETS)\simeq h\TOP.$$
\end{thm}

The adjoint categorical equivalence probably lifts to a direct Quillen equivalence between the test model category $\CUBICAL\SETS$ and $\TOP$ equipped with a suitable model structure, although we will not need such an equivalence in the current paper.  

\subsubsection{Cubical approximations}
We show that various uniform maps admit cubical approximations.
To begin, we show that the natural isomorphism $\varphi:\uniformrealization{\sd\,-}\cong\uniformrealization{-}$ naturally admits natural cubical approximations in the following sense.  

\begin{lem}
  \label{lem:lipschitz.approximation}
  There exist $1$-Lipschitz homotopies
  $$\geometricrealization{\firstvertexmap_C}\sim\geometricrealization{\lastvertexmap_C}\sim\varphi_C:\geometricrealization{\sd\,C}\ra\geometricrealization{C}$$
  natural in cubical sets $C$.  
\end{lem}
\begin{proof}
  Linear interpolations define the desired $1$-Lipschitz homotopies in the case $C$ representable because the metric on hypercubes $\I^n$ is defined in terms of a norm and hence in the general case by naturality.
\end{proof}

\begin{lem}
  \label{lem:uniform.approximation}
  Consider the following solid commutative diagrams
  \begin{equation}
    \label{eqn:approx}
    \xymatrix{
	**[l]\sd^{k}A\ar[rr]^{({\gamma}^{\mins(k-4)}(\firstvertexmap\lastvertexmap)^2)_A}\ar[d]_{\iota} & & **[r]A\ar[d]^{\alpha}\\
      **[l]\sd^kB\ar@{.>}[rr]_{\psi} & & C
    }\quad
    \xymatrix{
	  **[l]\uniformrealization{A}\ar[d]_{\uniformrealization{\iota}}\ar[r]^{\uniformrealization{\alpha}} & **[r]\uniformrealization{C}\\
       **[l]\uniformrealization{B}\ar[ur]_f
    }
  \end{equation}
  of cubical sets and uniform spaces with $B$ a coproduct of finitely many connected cubical sets.  
  For all $k\gg 0$, there exists a cubical function $\psi_f$ such that the left triangle commutes and uniform homotopy $h_f:\uniformrealization{\psi_f}\sim f\uniformrealization{\gamma^{\mins(k-4)}(\firstvertexmap\lastvertexmap)^2)_B}$, both natural in $(\uniformrealization{-}/\uniformrealization{-})$-objects $f$.  
  If $f$ defines a Lipschitz map $\geometricrealization{B}\ra\geometricrealization{C}$, then $h_f$ can be taken to be a Lipschitz homotopy from $\geometricrealization{\psi_f}$ to $\geometricrealization{(\hat\gamma^{+(k-4)}(\adjointfirstvertexmap\adjointlastvertexmap)^2)_C}f$.
\end{lem}

While several techniques in the proof mimic techniques in proofs of cubical approximation for directed topology \cite{krishnan2015cubical}, there are some differences.  
Firstly, uniform continuity makes it possible to drop finiteness constraints on the domain.  
Secondly, the construction of a cubical approximation preserving edge orientations follows from the properties of the subdivision operator instead of the directedness of a directed map.

\begin{proof}
  Take $k\gg 0$.
  Let $\sigma$ denote a $(\BOX[-]/\sd^kB)$-object.  
  For brevity, let 
  $$\gamma=\firstvertexmap\lastvertexmap:\sd^2\ra\id_{\CUBICAL\SETS}.$$

  For each $\epsilon>0$, there exists $\delta_\epsilon>0$ such that $|\gamma^2_{C}|(\varphi^{-4}_Cf\varphi^k_B)$ maps $\delta_\epsilon$-close points to $\epsilon$-close points by $f$ and hence $|\gamma^2_{C}|(\varphi^{-4}_Cf\varphi^k_B)$ uniform.
  Also $\varphi^{-4}_Cf\varphi^k_B$ maps each closed star of a vertex $v$ into an open star of a vertex by $f$ uniform [Lemma \ref{lem:uniform.characterization}].
  Therefore $|\gamma_{\sd^2C}|(\varphi^{-4}_Cf\varphi^k_B)$ maps each closed cell of $|\sd^kB|$ into a closed cell $A(\sigma)$ of $|\sd^2C|$ [Lemma \ref{lem:collapse.star}].
  We can take $A(\sigma)$ to be the minimal possible choice of a subpresheaf of $\sd^2C$, and hence natural in $\sigma$, because subpresheaves of $C$ that are images of representables are closed under non-empty intersection.  
  There exists unique minimal subpresheaf $B(\sigma)\subset C$, for each $\sigma$, for which $B(\sigma)$ is the image of a representable and $R(\sigma)=A(\sigma)\cap\sd^2B(\sigma)\neq\varnothing$ [Lemma \ref{lem:star.flower}].  
  There exists a unique retraction $\rho_\sigma:A(\sigma)\ra R(\sigma)$ and $R(\sigma)$ is isomorphic to a representable [Lemma \ref{lem:star.flower}].
  Take $f_\sigma$ to be the composite, natural in $\sigma$, of the corestriction of $|\gamma_{\sd^2C}|\varphi^{-4}_Cf\varphi^k_B\geometricrealization{\sigma}$ to $\geometricrealization{A(\sigma)}$ followed by $\geometricrealization{\rho_\sigma}$.
  Then $f_\sigma$ maps $\delta_\epsilon$-close points to $\epsilon$-close points by $\geometricrealization{\sigma_*}$ $1$-Lipschitz and the commutativity of the following diagram [Lemma \ref{lem:star.flower}]:
  \begin{equation*}
	  \begin{tikzcd}
		  {|\BOX\boxobj{n}|}\ar{rr}[above]{f_\sigma}\ar{d}[left]{\sigma} & &  {|R(\sigma)|}\ar{d}[right]{|\gamma_C(R(\sigma)\ira\sd^2C)|}\\
		  {|\sd^kB|}\ar{rr}[below]{|\gamma^2_{C}|(\varphi^{-4}_Cf\varphi^k_B)} & & {|C|}
	  \end{tikzcd}
  \end{equation*}

  Define a function $\phi_{f,\sigma}:(\BOX\boxobj{n})_0\ra\Tau_1R(\sigma)$, natural in $\sigma$, by the rule
  $$\phi_{f,\sigma}(v)=\support_{|-|}(f_\sigma|v_*|,R(\sigma))_0(0,\cdots,0).$$

  The function $\phi'_{f,\sigma}:(\sd\BOX\boxobj{n})_0\ra\Tau_1R(\sigma)$ defined by the rule
  $$\phi'_{f,\sigma}(v)=\min\left(\phi_{f,\sigma}((\firstvertexmap_{\BOX\boxobj{n}})_0(v)),\phi_{f,\sigma}((\lastvertexmap_{\BOX\boxobj{n}})_0(v))\right)$$
  is a monotone function $\Tau_1\sd\BOX\boxobj{n}\ra\boxobj{n(\sigma)}$ [Lemma \ref{lem:dilation.order}].
  Therefore $\phi'_{f,\sigma}$, natural in $\sigma$, admits an adjoint, natural in $\sigma$, of the form
  $$\psi_{f,\sigma}:\sd\BOX\boxobj{n}\ra\cnerve\Tau_1R(\sigma).$$

  There exists a homotopy $h_\sigma:|\BOX\boxobj{n(\sigma)}\ira\cnerve\boxobj{n(\sigma)}|f_\sigma\sim|\psi_{f,\sigma}|$ natural in $\sigma$ [Lemma \ref{lem:sharp.local.convexity}]. 
  For a fixed $t\in\I$, $h_{\sigma}(-,t)$ is a convex linear combination of $f_\sigma$ and a $1$-Lipschitz map and therefore maps $\min(\delta_\epsilon,\epsilon)$-close points to $\epsilon$-close points.  
  For a fixed $x\in|\BOX\boxobj{n}|=\I^n$, $h_{\sigma}(x,-)$ is $\|x\|_\infty$-Lipschitz and therefore $1$-Lipschitz.
  Therefore $h_\sigma$ maps $\min(\delta_\epsilon,\epsilon)$-close points to $\epsilon$-close points.  

  The $\psi_{f,\sigma}$'s induce a cubical function $\psi_f:\sd^kB\ra C^{\sharp}$.
  The $h_\sigma$'s induce a homotopy $h_f:|\gamma^2_{C}|(\varphi^{-4}_Cf\varphi^{k+1}_B)\sim|\psi_f|$ natural in $f$. 
  The function $h_f$ maps $\min(\delta_\epsilon,\epsilon)$-close points to $\epsilon$-close points because $\geometricrealization{\sigma}$ is $1$-Lipschitz.  
  Thus $h_f$ is uniform.  
  If $f$ is additionally Lipschitz, then $\delta_\epsilon$ can be taken to be linear in $\epsilon$; in this case $\min(\delta_\epsilon,\epsilon)$ is linear in $\epsilon$ and therefore $h_f$ is Lipschitz.

  It remains to show show that the left square commutes.  
  It suffices to take the case $A$ representable, $\iota=\id_{A}$ and hence $f=|\alpha|=|\id_{A}|$ by naturality.
  It therefore suffices to show that the square of vertex functions associated to the left square in (\ref{eqn:approx}) commutes, because cubical functions to representables $A=C$ are determined by their associated vertex functions.
  For each vertex $v$ in $\sd^{k}A$,
  \begin{align}
	     (\psi_{f})_0(v) 
	  &= \phi'_{f,\support_{\sd}(v_*,A)}(v)\\
	  &= \phi'_{|\alpha|,\support_{\sd}(v_*,A)}(v)\\
	  &= \support_{|-|}((|\gamma^2_C|\varphi^{-4}_C|\alpha|\varphi^{k}_A)(v),C)_0(0,\ldots,0)\\
	  &= \support_{|-|}((|\gamma^2_C|\varphi^{k-4}_{\sd^2C})(v_*),C)_0(0,\ldots,0)\\
	  &= (\gamma^2_C)_0\support_{|-|}((\varphi^{k-4}_{\sd^2C})(v_*),\sd^2C)_0(0,\ldots,0)\\
	  &=\label{eqn:first.vertex.map} (\gamma^2_C(\gamma^{\mins(k-3)}_{\sd^2C})_0(v)\\ 
	  &= ((\gamma^{\mins(k-4)}\gamma^2)_C)_0(v)\\ 
	  &= (\alpha\gamma^{\mins(k-4)})_{A}\gamma^2_A)_0(v)
  \end{align}
  the equality (\ref{eqn:first.vertex.map}) by Lemma \ref{lem:dilation.formula}.
\end{proof}

We can show that the adjoint $B\ra\ex^kC$ to a cubical approximation $\sd^kB\ra C$ of a uniform map $|B|\ra|C|$ is also a cubical approximation as follows.  

\begin{lem}
  \label{lem:adjoint.approximation}
  Consider the following solid commutative diagrams
  \begin{equation}
    \label{eqn:approx}
    \xymatrix{
	**[l]A\ar[r]^{\alpha}\ar[d]_{\iota} & **[r]C\ar[d]^{(\hat{\gamma}^{+(k-4)}(\adjointfirstvertexmap\adjointlastvertexmap)^2)_C}\\
      **[l]B\ar@{.>}[r]_{\psi_f} & **[r]\ex^kC
    }\quad
    \xymatrix{
	  **[l]\uniformrealization{A}\ar[d]_{\uniformrealization{\iota}}\ar[r]^{\uniformrealization{\alpha}} & **[r]\uniformrealization{C}\\
       **[l]\uniformrealization{B}\ar[ur]_f
    }
  \end{equation}
  of cubical sets and uniform spaces, with $B$ a coproduct of fintely many connected cubical sets.  
  For all $k\gg 0$, there exists a cubical function $\psi_f$ such that the left triangle commutes and uniform homotopy $h_f:\uniformrealization{\psi_f}\sim\uniformrealization{\hat\gamma^{+(k-4)}(\adjointfirstvertexmap\adjointlastvertexmap)^2)_C} f$.
  The homotopy $h_f$ can be taken to be relative $|\iota|$ if $\iota$ is a monoidal product of an identity with a monic cubical function between finite cubical sets.  
  If $f$ defines a Lipschitz map $\geometricrealization{B}\ra\geometricrealization{C}$, then $h_f$ can be taken to be a Lipschitz homotopy from $\geometricrealization{\psi_f}$ to $\geometricrealization{(\hat\gamma^{+(k-4)}(\adjointfirstvertexmap\adjointlastvertexmap)^2)_C}f$.
\end{lem}

We will use the formal fact that in a $\TOP$-enriched category $\homotopicalcat{1}$ tensored and cotensored over $\TOP$, a homotopy $x\otimes\I\ra y$ between maps $\alpha,\beta:x\ra y$ for which $\alpha\iota=\beta\iota$ can be taken to be relative $\iota$ if $\iota$ is an h-cofibration.   

\begin{proof}
  Fix $k\gg 0$.  
  Let $\eta$ denote the unit of $\sd^k\dashv\ex^k$.  
  Let
  $$\gamma=\gamma^{\mins (k-4)}(\firstvertexmap\lastvertexmap)^2\quad\hat\gamma=\hat\gamma^{\mins (k-4)}(\adjointfirstvertexmap\adjointlastvertexmap)^2.$$
  There exists a cubical function $\psi'_f$ making the top parallelogram in
\begin{equation*}
	\begin{tikzcd}
		{\geometricrealization{B}}
		      \ar{rr}[above]{f} 
			  \ar{dd}[left]{\geometricrealization{\eta_B}}
		& 
		& {\geometricrealization{C}}
		      \ar[dr,equals]
		\\
		& {\geometricrealization{\sd^kB}}
		\ar{ul}[description]{\geometricrealization{\gamma_B}}
		  \ar[dotted]{rr}[description]{\geometricrealization{\psi'_f}} 
			\ar{dl}[description]{\hat\gamma_{\sd^kB}}
		& 
		& {\geometricrealization{C}}
		  \ar{dl}[description]{\geometricrealization{\hat\gamma_C}}\\
		{\geometricrealization{\ex^k\sd^kB}}
		  \ar{rr}[below]{\geometricrealization{\ex^k\psi'_f}}
		&
		& {\geometricrealization{\ex^kC}}
	\end{tikzcd}
\end{equation*}
commute up to uniform homotopy, and in fact Lipschitz homotopy if $f$ is Lipschitz [Lemma \ref{lem:uniform.approximation}].
  The left triangle and bottom parallelogram commute by naturality.
  Moreover, the arrow $\geometricrealization{\gamma_B}$ admits an inverse up to Lipschitz homotopy [Lemma \ref{lem:lipschitz.approximation}].  
  Let $\psi_f=(\ex^k\psi'_f)\eta_B$.
  Then there exists a uniform homotopy of the following form, Lipschitz if $f$ is Lipschitz: $$\geometricrealization{\hat\gamma_C}f\sim\geometricrealization{\psi_f}.$$
 
  The cubical function $\psi'_f$ can be chosen so that the following holds [Lemma \ref{lem:uniform.approximation}]:
  $$\psi'_f(\sd^k\iota)=\alpha\gamma_A$$ .  
  The left square therefore commutes because
  $$\psi_f\iota=(\ex^k\psi'_f)\eta_B\iota=(\ex^k(\alpha\gamma_A))\eta_A=\hat\gamma_C\alpha.$$

  Suppose $\iota$ is a monoidal product of an identity with a monic cubical function between finite cubical sets. 
  The uniform map $\uniformrealization{\iota}$, a Cartesian monoidal product of an identity with a Hurewicz cofibration between compact Hausdorff spaces, is a uniform Hurewicz cofibration.
  Therefore the uniform homotopy $h_f$ can be taken to be relative $|\iota|$.   
\end{proof}

\subsubsection{Anodyne Extensions}
We next show that natural inclusions of the form
$$C\ira\ex^kC^{\sharp}$$
are Lipschitz homotopy equivalences after passing to geometric realizations and therefore uniform homotopy equivalences after passing to uniform realizations.  

\begin{lem}
  \label{lem:sharp.replacement}
  For each cubical set $C$, the $1$-Lipschitz map 
  $$\geometricrealization{C\ira C^{\sharp}}:\geometricrealization{C}\ra\geometricrealization{C^{\sharp}}$$
  is a Lipschitz homotopy equivalence.
\end{lem}
\begin{proof}
  There exists a $1$-Lipschitz retraction $\rho_C$ to $\geometricrealization{C\ira C^{\sharp}}$ natural in $C$ [Lemma \ref{lem:geometric.sharp.retraction}].  
  It therefore suffices to show that there exists a Lipschitz homotopy from $\geometricrealization{C\ira C^{\sharp}}r_{C}$ to $\id_{\geometricrealization{C^{\sharp}}}$ natural in cubical sets $C$.  
  Therefore it suffices to take the case $C=\BOX\boxobj{n}$ representable by naturality.  
  In that case, $\geometricrealization{C\ira C^{\sharp}}r_{C}$ is a Lipschitz map to $\geometricrealization{\cnerve\boxobj{n}}$.
  There exists a Lipschitz homotopy, natural in $\BOX$-objects $\boxobj{n}$, between projections of the form $(\geometricrealization{\cnerve\boxobj{n}})^2\ra\geometricrealization{\cnerve\boxobj{n}}$ [Lemma \ref{lem:sharp.local.convexity}].  
  The desired Lipschitz homotopy follows.
\end{proof}

\begin{lem}
  \label{lem:extensions}
  For each cubical set $C$ and each $k$, the $1$-Lipschitz map
  \begin{equation}
    \label{eqn:extension}
	\geometricrealization{\hat{\gamma}^{+k}_{C^{\sharp}}}:\geometricrealization{C^{\sharp}}\ira\geometricrealization{\ex^k C^{\sharp}}
  \end{equation}
  is a Lipschitz homotopy equivalence.
\end{lem}
\begin{proof}
  Let $\epsilon$ denote the counit of the adjunction $\sd^k\dashv\ex^k$.
  Let 
  $$F=\geometricrealization{\ex^k-}:\CUBICAL\SETS\ra\METRIC.$$

  The $1$-Lipschitz map (\ref{eqn:extension}) admits, up to Lipschitz homotopy, a Lipschitz retraction
  $$r_C=\geometricrealization{\epsilon^{k}_{\C^{\sharp}}}\varphi^{-k}_{\ex^k\C^{\sharp}}:\geometricrealization{\ex^k\C^{\sharp}}\ra\geometricrealization{\C^{\sharp}}$$
  natural in cubical sets $C$ [Lemma \ref{lem:lipschitz.approximation}].  
  It thus suffices to construct a Lipschitz homotopy
  $$\geometricrealization{\hat{\gamma}^{+k}_{\C^{\sharp}}}r_{C}\sim\id_{FC}$$
  natural in cubical sets $C$.
  Therefore it suffices to consider the case $C$ representable by naturality.
  There exists a Lipschitz homotopy, natural in $\BOX$-objects $\boxobj{n}$, between both projections of the form $(F\cnerve\boxobj{n})^2\ra(F\cnerve\boxobj{n})$ [Lemma \ref{lem:sharp.local.convexity}].
  The existence of the desired Lipschitz homotopy follows.
\end{proof}

\subsubsection{Main results}
Putting together the results of the previous two sections yields a formal comparisom between the Lipschitz homotopy theory of cubical sets and the uniform homotopy theory of uniform spaces.

\begin{thm}
  \label{thm:enriched.equivalence}
  There exists a classical weak equivalence
  $$\langle B,C\rangle\ra\sing_{\BOX}\UNIFORM(|B|_\infty,|C|_\infty)$$
  natural in cubical sets $B,C$.  
\end{thm}
\begin{proof}
  For brevity, adopt the abbreviation
  $$S_{BC}=\sing_\BOX\UNIFORM(|B|_\infty,|C|_\infty).$$
  
  We can henceforth make natural identifications
  $$\CUBICAL\SETS(B,\ex^kC)_n=\CUBICAL\SETS(\sd^k(B\otimes\BOX\boxobj{n}),C)\quad (S_{BC})_n=\UNIFORM(|B|_\infty\times|\BOX\boxobj{n}|_\infty,|C|_\infty)$$
  by the Yoneda Embedding and the adjoint relationships $|-|\dashv\sing_\BOX$ and $\sd\dashv\ex$.  

  It suffices to take $B$ connected.  
  Consider the following diagram
  \begin{equation*}
    \begin{tikzcd}
		  \CUBICAL\SETS(B,C)\ar[r]\ar{d}[description]{\rho_{B,C,1}} 
		& \CUBICAL\SETS(B,\ex\,C)\ar[r]\ar{dl}[description]{\rho_{B,C,2}} 
		& \CUBICAL\SETS(B,\ex^2C)\ar[r]\ar{dll}[description]{\rho_{B,C,3}}  
		& \cdots
		& \langle B,C\rangle\ar[dotted]{d}[description]{\rho_{B,C}}\\
		  S_{BC}\ar[rrrr,equals]
		&
		&
		&
		& S_{BC}
	\end{tikzcd}
  \end{equation*}
  where $(\rho_{B,C,k})_n(\sigma)=|\sigma|_\infty\varphi^{k}_{B\otimes\BOX\boxobj{n}}$ and the top horizontal arrows are induced by $\gamma^{\mins}_{\ex^kC}$. 
  The solid diagram commutes up to cubical homotopy [Lemma \ref{lem:lipschitz.approximation}].  
  Therefore the dotted cubical functions induce a cubical function $\rho_{B,C}$ from the colimit of the top row, $\langle B,C\rangle$, a homotopy colimit in the classical test model structure on $\CUBICAL\SETS$ by the solid arrows monic, making the entire diagram commute up to cubical homotopy. 
  The $|\rho_{B,C,k}|$'s induce bijections on path-components and homotopy groups by an application of cubical approximation [Lemma \ref{lem:adjoint.approximation}].  
  Therefore $|\rho_{B,C}|$, and hence also $\rho_{B,C}$, are classical weak equivalences.
\end{proof}

\begin{cor}
  \label{cor:cubical.we}
  Fix a cubical function $\psi:B\ra C$.
  The following are equivalent.
  \begin{enumerate}
	  \item $\psi$ is a uniform weak equivalence
	  \item $\uniformrealization{\psi}$ is a uniform homotopy equivalence
	  \item $\uniformrealization{\psi}$ is a uniform weak equivalence
	  \item $\geometricrealization{\psi}$ restricts and corestricts to Lipschitz homotopy equivalences between connected components
  \end{enumerate}
\end{cor}
\begin{proof}
  It suffices to take $B,C$ connected.  

  Suppose (1).
  Then $\pi_0\langle C,\psi\rangle$ is a bijection and hence there exists a cubical function $\psi^\dagger:\sd^kC\ra B$ such that $\psi^\dagger(\sd^k\psi)$ and $\gamma^{-k}_B$ represent the same element in $\pi_0\langle B,B\rangle$.  
  It follows that there exists a Lipschitz homotopy $\geometricrealization{\psi^\dagger}\geometricrealization{\sd^k\psi}\sim\geometricrealization{\gamma^{-k}_B}$ and hence there exists a Lipschitz homotopies $\geometricrealization{\psi^\dagger}\varphi_C^{-k}\sim\geometricrealization{\psi}\sim\geometricrealization{\gamma^{-k}_B}\varphi_B^{-k}\sim\geometricrealization{\id_B}$.  
  Let $\mathscr{H}$ denote the category of pseudometric spaces and Lipschitz homotopy classes of Lipschitz maps.  
  Thus $\geometricrealization{\psi}$ represents a morphism in $\mathscr{H}$ admitting a retraction.
  Similarly $\geometricrealization{\psi}$ represents a morphism in $\mathscr{H}$ admitting a section.  
  Therefore $\geometricrealization{\psi}$ represents an $\mathscr{H}$-isomorphism.
  Hence (4).  

  Suppose (4). 
  Lipschitz maps are uniform.
  Therefore (2).  

  Suppose (2).  
  Then (3) [Proposition \ref{prop:strong.to.weak.uniform.equivalences}].  

  Suppose (3).  
  Fix a cubical set $A$. 
  There are dotted classical weak equivalences making
  \begin{equation*}
	  \begin{tikzcd}
		          \langle A,B\rangle\ar{rrrr}[above]{\langle A,\psi\rangle}\ar[d,dotted] 
		  & & & & \langle A,C\rangle\ar[d,dotted]
		  \\ 
		          \sing_{\BOX}\UNIFORM(|A|_\infty,|B|_\infty)\ar{rrrr}[below]{\sing\UNIFORM(|A|_\infty,|\psi|_\infty)} 
		  & & & & \sing_{\BOX}\UNIFORM(|A|_\infty,|C|_\infty)
	  \end{tikzcd}
  \end{equation*}
  commute [Theorem \ref{thm:enriched.equivalence}].
  The bottom horizontal arrow is a classical weak equivalence by assumption.
  Therefore the top horizontal arrow is a classical weak equivalence.
  Hence (1).
\end{proof}

\begin{cor}
  \label{cor:classical.we}
  The following are equivalent for a cubical function $\psi$.
  \begin{enumerate}
    \item $\psi$ is a classical weak equivalence
	\item $|\psi|$ is a classical weak equivalence
	\item $|\psi|$ is a classical homotopy equivalence
	\item $\sing_{\BOX}\psi$ is a cubical homotopy equivalence
  \end{enumerate}
  For all cubical sets $C$, $\langle \star,C\rangle$ is a Kan complex and the natural inclusion $C\ira\langle\star,C\rangle$ is a classical weak equivalence.
\end{cor}
\begin{proof}
  The equivalence (1)$\iff$(2) follows from Theorem \ref{thm:equivalence}.
  The equivalence (2)$\iff$(3) follows from the Whitehead Theorem.  
  The equivalence (3)$\iff$(4) follows from the adjunction $|-|\dashv\sing_{\BOX}$.

  Consider the left of the solid diagrams
  \begin{equation}
    \label{eqn:horn.approx}
	\begin{tikzcd}
	\sqcup^i\BOX\boxobj{n}\ar[r]\ar[d,hookrightarrow] & \langle\star,C\rangle\\
	\BOX\boxobj{n}\ar[ur,dotted]
    \end{tikzcd}\quad
	\begin{tikzcd}
		{|\sqcup^i\BOX\boxobj{n}|}\ar[r]\ar[d,hookrightarrow] & {|\ex^kC|}\\
		{|\BOX\boxobj{n}|}\ar[ur,dotted]
    \end{tikzcd}
   \end{equation}
  There exists a continuous, and hence uniform by $|\BOX\boxobj{n}|$ compact Hausdorff, retraction to $\uniformrealization{\sqcup^i\boxobj{n}\ira\BOX\boxobj{n}}$. 
  Therefore there exists a dotted uniform map making the right triangle in (\ref{eqn:horn.approx}) commute.
  Therefore $C$ is Kan because there exists a dotted cubical function making the left triangle commute [Lemma \ref{lem:uniform.approximation}].  
  The natural inclusion $C\ira\langle\star,C\rangle$, a transfinite composite of homotopy equivalences after applying topological realization [Lemma \ref{lem:extensions}], is a transfinite composite of monic classical weak equivalences [Theorem \ref{thm:equivalence}] and hence a classical weak equivalence.  
\end{proof}

\begin{cor}
  \label{cor:embedding}
  \COREmbedding{}
\end{cor}

We now give a proof that Lispchitz weak equivalences are classical weak equivalences. 

\begin{proof}[proof of Proposition \ref{prop:lipschitz.implies.classical}]
  For a Lipschitz weak equivalence $\psi:B\ra C$, 
  $$\langle\star,\psi\rangle:\langle\star,B\rangle\ra\langle \star,C\rangle$$
  is a classical weak equivalence of Kan cubical sets [Corollary \ref{cor:classical.we}] and hence a cubical homotopy equivalence.  
  Therefore for each cubical set $A$, $\langle\star,\psi\rangle^A$ is a cubical homotopy equivalence and in particular a bijection after applying $\pi_0$.  
\end{proof}

We can then give a proof that Lipschitz and classical weak equivalences coincide on cubical functions between finite cubical sets.  

\begin{proof}[proof of Proposition \ref{prop:finite.lipschitz}]
  Suppose (1).  
  It suffices to show (2) [Proposition \ref{prop:lipschitz.implies.classical}].  
  Then $|\psi|$ is a classical homotopy equivalence of spaces.  
  Therefore $\uniformrealization{\psi}$ is a uniform homotopy equivalence because $\uniformrealization{A}\times\I,\uniformrealization{B}\times\I$ are compact Hausdorff uniform spaces and therefore compact Hausdorff spaces equipped with their fine uniformities.  
  Therefore $\psi$ is a Lipschitz weak equivalence [Corollary \ref{cor:classical.we}].  
\end{proof}

We now give a proof that cubical sets is a category of cofibrant objects whose weak equivalences are the Lipschitz weak equivalences and whose cofibrations are the monos.  

\begin{proof}[proof of Proposition \ref{prop:cofibrant.objects}]
  The category $\CUBICAL\SETS$, a presheaf category, has all finite coproducts.  
  Lipschitz weak equivalences satisfy the 2-of-out-3 property because classical weak equivalences satisfy the 2-out-of-3 property. 
  Each codiagonal $C\amalg C\ra C$ factors as a mono $C\amalg C\ira C\otimes\BOX[1]$ followed by a cubical homotopy equivalence and hence Lipschitz weak equivalence $C\otimes(\BOX[1]\ra\BOX[0]):C\otimes\BOX[1]\ra C$ [Proposition \ref{prop:strong.to.weak.lipschitz.equivalences}].
  Every initial cubical function $\varnothing\ra C$ is monic.  
  Monos are closed under pushouts in presheaf categories like $\CUBICAL\SETS$.
  
  Consider a monic Lipschitz weak equivalence $\psi:A\ra B$.
  Then $\geometricrealization{\psi}$ is a Lipschitz homotopy equivalence [Corollary \ref{cor:cubical.we}].
  Then every pushout of $\psi$, monic because monos are closed under pushouts, is a Lipschitz homotopy equivalence after taking geometric realizations because $\geometricrealization{-}$ preserves pushouts and pushouts preserve Lipschitz homotopy equivalences.
  Therefore every pushout of $\psi$, monic because pushouts preserve monos in presheaf categories, is also a Lipschitz weak equivalence [Corollary \ref{cor:cubical.we}].  

  It follows that $\CUBICAL\SETS$ is a category of cofibrant objects with the desired weak equivalences and cofibrations.  
  Consider a cubical set $B$.  
  If $B$ is Kan, then every cubical function $A\ra B$ extends along every monic classical weak equivalence from $A$, and in particular every monic Lipschitz weak equivalence from $A$ [Proposition \ref{prop:lipschitz.implies.classical}], and hence $B$ is fibrant. 
  If $B$ is fibrant, then every cubical function $A\ra B$ extends along every monic Lipschitz weak equivalence from $A$, including inclusions of the form $\sqcup^{\pm i}\BOX\boxobj{n}\ira\BOX\boxobj{n}$ for $A=\sqcup^{\pm i}\BOX\boxobj{n}$[Proposition \ref{prop:finite.lipschitz}], and hence $B$ is Kan.  
\end{proof}

We now give a proof that the hom-sets in $h_\infty(\CUBICAL\SETS)$ are of the form $\pi_0\langle-,-\rangle$.

\begin{proof}[proof of Proposition \ref{prop:lipschitz.weak.homotopy.classes}]
  There exist natural bijections of the form
  \begin{align*}
	  [B,C] 
	  &\cong h_\infty\UNIFORM(\uniformrealization{B},\uniformrealization{C}) && [\mathrm{Corollary}\,\ref{cor:embedding}]\\
	  &\cong \pi_0\sing_{\BOX}\UNIFORM(\uniformrealization{B},\uniformrealization{C}) && [\mathrm{Proposition}\,\ref{prop:cofibrant.uniform.spaces}]\\
	  &\cong \pi_0\langle B,C\rangle && [\mathrm{Theorem}\,\ref{thm:enriched.equivalence}]
  \end{align*}
\end{proof}

\section{Cohomology}\label{sec:cohomology}
We introduce cubical and uniform variants of bounded cohomology, the former of which generalizes bounded group cohomology [Example \ref{eg:bounded.group.cohomology}] and the latter of which generalizes bounded singular cohomology [Example \ref{eg:bounded.singular.cohomology}].  
The main result in this section is the representability of cubical bounded cohomology on a large class of cubical sets by a surgical construction on cubical models of Eilenberg-Maclane spaces [Theorem \ref{thm:cubical.cohomology.representable}].
We can then obtain a representable theory on uniform spaces that generalizes bounded singular cohomology on path-connected spaces [Corollary \ref{cor:singular.cohomology.representable}].  
A special case of the embedding of the Lipschitz homotopy category into the uniform homotopy category is that our representable cubical theory coincides with our representable uniform theory [Proposition \ref{prop:cohomological.comparison}].
Along the way, the weak Lipschitz invariance of $\ell_p$ cubical cohomology is investigated [Example \ref{eg:lipschitz.invariance}, Proposition \ref{prop:lipschitz.invariance}].

\subsection{Cubical}
We give a straightforward definition of $\ell_p$-cohomology as the cohomology of the $\ell_p$ cellular cochains on topological realizations; investigate the invariance of $\ell_p$ cohomology [Proposition \ref{prop:lipschitz.invariance}]; and additionally investigate the representability of $\ell_\infty$ cohomology [Theorem \ref{thm:cubical.cohomology.representable}].

\subsubsection{Bounded cohomology}
There exist dotted homomorphisms making
\begin{equation}
  \label{eqn:cubical.cochains}
	\begin{tikzcd}
		\ell_p(C_0,\pi)\ar{r}[above]{\partial^0_{\ell_p}}\ar{d}[description]{\ell_p(sC_{-1}\ira C_0,\pi)}
	  &	\ell_p(C_1,\pi)\ar{r}[above]{\partial^{1}_{\ell_p}}\ar{d}[description]{\ell_p(sC_{0}\ira C_1,\pi)} 
	  & \ell_p(C_2,\pi)\ar{d}[description]{\ell^p(sS_{1}\ira S_2,\pi)}\ar[r]
	  & \cdots
      \\
	    \ell_p(sC_{-1},\pi)\ar[r,dotted]
	  & \ell_p(sC_0,\pi)\ar[r,dotted]
	  & \ell_p(sC_{1},\pi)\ar[r,dotted]
	  & \cdots
	\end{tikzcd}
\end{equation}
commute and natural in cubical sets $C$, where 
$$(\partial^n_{\ell_p}c)(\sigma)=\sum_{i=1}^n(-1)^ic(d_{+i}(\sigma))-c(d_{-i}(\sigma))$$

Write $\chains^*_{\ell_p}(C;\pi)$ for the kernel of the above diagram, regarded as a map from the top connective cochain complex to the bottom connective cochain complex.
Let 
$$\lpcohomology{*}{p}(C;\pi)=H^n\chains^*_{\ell_p}(C;\pi).$$ 

For Abelian groups $\pi$ (regarded as equipped with a constant seminorm), write $\chains^*(C;\pi)$ for  $\chains^*_{\ell_\infty}(C;\pi)$ and $H^*(C;\pi)$ for $\boundedcohomology{*}(C;\pi)$.
The identity function on the underlying Abelian group of a seminormed Abelian group $\pi$ induces a natural \textit{comparison} map
$$\lpcohomology{*}{p}(C;\pi)\ra H^*(C;\pi).$$

\subsubsection{Lipschitz invariance}
We now investigate the invariance of $\ell_p$ cohomology.
Cubical homotopy equivalences $A\simeq B$ induce chain homotopy equivalences $\chains^*_{\ell_p}(A;\pi)\simeq\chains^*_{\ell_p}(B;\pi)$ and therefore isomorphisms $\lpcohomology{*}{p}(A;\pi)\cong\lpcohomology{*}{p}(B;\pi)$.

\begin{eg}
  \label{eg:bounded.group.cohomology}
  Fix a group $G$. 
  Observe
  $$\boundedcohomology{*}(G;\R)\cong\boundedcohomology{*}(\sing_{\BOX}|\cnerve G|;\R)\cong\boundedcohomology{*}(\cnerve G;\R),$$
  where $\boundedcohomology{*}(G;\R)$ denotes the \textit{bounded group cohomology} of a group $G$ with coefficients in the Abelian group of additive reals $\R$ equipped with the usual norm and trivial $\R[G]$-action.
  The first isomorphism is classical \cite{gromov1982volume} after identifying the middle group with the usual kind of bounded singular cohomology based on semi-simplicial sets [Example \ref{eg:bounded.singular.cohomology}].
  The second isomorphism follows because $\sing_{\BOX}|\cnerve G|,\cnerve G$ are classically weakly equivalent and hence cubically homotopy equivalent by fibrancy in the test model structure.
\end{eg}

Ordinary cubical cohomology is representable on the classical homotopy category of cubical sets by the Dold-Kan correspondence and therefore a weak homotopy invariant.
Bounded cubical cohomology, while cubically homotopy invariant, is generally not weak Lipschitz invariant.

\begin{eg}
  \label{eg:lipschitz.invariance}
  Define a cubical set $C$ as the colimit of the solid diagram
  \begin{equation*}
    \begin{tikzcd}
        \BOX[0]\ar[r] & C\\
		\amalg_{i=1}^\infty\BOX[0]\ar[u]
		\ar[r,shift right]
		\ar[r,shift left]
	 &  \amalg_{i=1}^\infty\BOX[1]\ar[u]
	\end{tikzcd}
  \end{equation*}
  where the parallel arrows are induced by cubical functions of the respective forms $\BOX[0\mapsto(0,\ldots,0)]:\BOX[0]\ra\BOX\boxobj{n}$ and  $\BOX[0\mapsto(1,\ldots,1)]:\BOX[0]\ra\BOX\boxobj{n}$.
  Intuitively, $C$ is an infinite wedge sum of ever larger cubical models of circles.
  Then $\boundedcohomology{1}(C\ira C^\sharp;\Z)$ is not an isomorphism because the $1$-cocycle in $\chains^*_{\ell_\infty}(C;\Z)$ assigning to each non-degenerate $1$-cube of $C$ the number $1$ extends to a representative of a non-trivial $1$-cohomology class in $H^1(C;\Z)$ but does not extend to a $1$-cochain on $\chains^*_{\ell_\infty}(C^\sharp;\Z)$.
\end{eg}

In general, $\ell_p$ cubical cohomology is subdivision invariant.

\begin{lem}
  \label{lem:sd.invariance}
  For each cubical set $C$ and seminormed Abelian group $\pi$,
  \begin{equation}
    \label{eqn:cohomological.subdivision}
    \lpcohomology{n}{p}(\gamma^{\pm}_C;\pi):\lpcohomology{n}{p}(C;\pi)\cong\lpcohomology{n}{p}(\sd\,C;\pi).
  \end{equation}
\end{lem}

The proof assumes familiarity with the method of acyclic models \cite{eilenberg1953acyclic}.

\begin{proof}
  Take the \textit{$\ell_1$-norm} of an element $\sum_\sigma\lambda_\sigma\sigma\in\chains_mB$ to be
  $$\|\sum_\sigma\lambda_\sigma\sigma\|_1=\sum_\sigma|\lambda_\sigma|.$$

  It suffices to take $C$ $(n+1)$-dimensional by definition of $n$th $\ell_p$ cohomology and dimension-preservation of $\sd$.
  
  The functor $\chains_*\sd$ is free on all cubical sets.
  The functor $\chains_*$ is acyclic on all representable cubical sets.  
  Therefore there exists a chain homotopy inverse $\Gamma^\pm_B$ to $\chains_*\gamma^{\pm}_B$ and chain homotopies $(\chains_*\gamma^{\pm}_B)\Gamma^\pm_B\sim\id_{\chains_*B}$ and $\Gamma^\pm_B(\chains_*\gamma^{\pm}_B)\sim\id_{\chains_*\sd\,B}$ natural in cubical sets $B$ by an acyclic models argument.  
  Write $(\Gamma^\pm_C)^*$ for the following cochain map induced by $\Gamma^\pm_C$:
  $$(\Gamma^\pm)_C^*:\chains^*(\sd\,C;\pi)\ra\chains^*(C;\pi).$$

  The chain map $\Gamma^\pm_B$ and the aforementioned chain homotopies all map cells to elements with uniformly bounded $\ell_1$-norms for the case $B=\BOX\boxobj{0},\cdots,\BOX\boxobj{n+1}$ by finiteness and hence for the case $B=C$ by naturality. 
  Therefore $(\Gamma^\pm_C)^*$ restricts and corestricts to a cochain homotopy inverse to $\chains^*_{\ell_p}(\gamma^{\pm}_C;\pi)$.
  Hence (\ref{eqn:cohomological.subdivision}).  
\end{proof}

On sharp and connected cubical sets, $\ell_p$ cohomology is invariant under sharp replacement.

\begin{lem}
  \label{lem:sharp.invariance}
  For each connected sharp cubical set $C$ and seminormed Abelian group $\pi$,
  $$\lpcohomology{*}{p}(C\ira C^\sharp;\pi):\lpcohomology{*}{p}(C^\sharp;\pi)\cong\lpcohomology{*}{p}(C;\pi).$$
\end{lem}
\begin{proof} 
  For all cubical functions $\psi:A\ra B$, let $\psi_*$ denote the induced map 
  $$\psi_*=\lpcohomology{*}{p}(\psi;\pi):\lpcohomology{*}{p}(A;\pi)\ra\lpcohomology{*}{p}(B;\pi).$$
  Let $\iota$ be inclusion $C\ira C^\sharp$.  
  Let $\mathscr{H}$ be the category of pseudometric spaces and Lipschitz homotopy classes of maps between them.  
  
  There exists a retraction $\rho$ to $C\ira C^\sharp$ by $C$ sharp. 
  The $1$-Lipschitz map $\geometricrealization{\iota}$ represents an $\mathscr{H}$-isomorphism [Lemma \ref{lem:extensions}] with retraction, an inverse, represented by $\geometricrealization{\rho}$.  
  Thus there exists a Lipschitz homotopy $\geometricrealization{\iota\rho}\sim\id_{\geometricrealization{C^\sharp}}$.
  Hence for $k\gg 0$ the following diagram commutes up to cubical homotopy by cubical approximation [Lemma \ref{lem:uniform.approximation}].  
  \begin{equation*}
    \begin{tikzcd}
	          C^\sharp\ar{rrrr}[above]{\id_{C^\sharp}} 
	  & & & & C^\sharp\\
	          \sd^kC^\sharp\ar{u}[left]{({\gamma}^{\mins(k-4)}(\firstvertexmap\lastvertexmap)^2)_{C^\sharp}}
	          \ar{rrrr}[below]{({\gamma}^{\mins(k-4)}(\firstvertexmap\lastvertexmap)^2)_{C^\sharp}} 
	  & & & & C^\sharp\ar{u}[right]{\iota\rho}
	\end{tikzcd}
  \end{equation*}
  The left vertical and bottom horizontal cubical functions induce isomorphisms on $\ell_p$ cohomology [Lemma \ref{lem:sd.invariance}].  
  Thus $\rho_*$, a section to $\iota_*$ by $\rho$ a retraction to $\iota$, is also a retraction to $\iota_*$.
\end{proof}

It then follows that $\ell_p$ cubical cohomology, a cubical homotopy invariant and therefore weakly Lipschitz invariant on connected Kan cubical sets [Proposition \ref{prop:cofibrant.objects}], is actually weakly Lipschitz invariant on more general sharp and connected cubical sets.   

\begin{prop}
  \label{prop:lipschitz.invariance}
  For a Lipschitz weak equivalence $\psi:A\ra B$,
  $$\lpcohomology{*}{p}(\psi;\pi):\lpcohomology{*}{p}(A;\pi)\ra\lpcohomology{*}{p}(B;\pi)$$ 
  is an isomorphism for all seminormed Abelian groups $\pi$ if $A,B$ are both sharp and connected. 
\end{prop}
\begin{proof}
  Consider the following commutative diagram of solid arrows.
  \begin{equation*}
    \begin{tikzcd}
		A\ar{rr}[above]{\psi}\ar[d,hookrightarrow] & & B\ar[d,hookrightarrow]\\
		A^\sharp\ar{rr}[description]{\psi^\sharp} & & B^\sharp\\
		\sd^k A^\sharp\ar{u}[left]{({\gamma}^{\mins(k-4)}(\firstvertexmap\lastvertexmap)^2)_{A^\sharp}}
		  \ar{rr}[below]{\sd^k\psi^\sharp} & & \sd^kB^\sharp\ar{u}[right]{({\gamma}^{\mins(k-4)}(\firstvertexmap\lastvertexmap)^2)_{B^\sharp}}\ar[ull,dotted]
	\end{tikzcd}
  \end{equation*}
  All of the vertical induces isomorphisms on $\ell_p$ cohomology [Lemmas \ref{lem:sd.invariance}, \ref{lem:sharp.invariance}].  
  And all of the arrows induce Lipschitz homotopy equivalences between geometric realizations of connected cubical sets [Lemma \ref{lem:lipschitz.approximation}, Corollary \ref{cor:cubical.we}].  
  Thus for $k\gg 0$ there exists a dotted cubical function making the entire diagram commute up to cubical homotopy by cubical approximation [Lemma \ref{lem:uniform.approximation}].  
  It therefore follows that all of the arrows induce isomorphisms on $\ell_p$ cohomology.  
\end{proof}

\subsubsection{Representability}
Bounded cohomology is generally not representable.

\begin{eg}
  \label{eg:non-representability}
  Bounded cohomology does not send disjoint unions to products:
  $$\boundedcohomology{0}\left(\amalg_{i=1}^\infty\BOX[0];\R\right)=\ell_\infty(\N,\R)\neq\prod_i\R=\prod_i\boundedcohomology{0}(\BOX[0];\R).$$
  Therefore $\boundedcohomology{0}(-;\pi)$ is not representable in the Lipschitz homotopy category, even for sharp cubical sets like $\amalg_{i=1}^\infty\BOX[0]$.  
\end{eg}

We will show that cubical cohomology can be expressed in terms of representable functors on the Lipschitz homotopy category as follows.  
For each Abelian group $\pi$, define $C(\pi,n)_m$ to be the following Abelian group natural in $\BOX$-objects $\boxobj{m}$:
$$C(\pi,n)_m=\CONNECTIVECHAINCOMPLEXES\left(\chains_*\cnerve\boxobj{m},\Sigma^n\pi\right).$$

\begin{eg}
  For each Abelian group $\pi$, there exists a cubical homotopy equivalence
  $$\cnerve\,\pi\simeq C(\pi,1).$$
\end{eg}

Fix a seminormed Abelian group $\pi$.  
We represent $\boundedcohomology{*}(-;\pi)$ in certain cases as follows.   
As a cubical set, $C(\pi,n)$ admits a filtration whose $i$th stage is the subpresheaf of $C(\pi,n)$ whose $m$-cubes are all chain maps $\chains_*\cnerve\boxobj{m}\ra\Sigma^n\pi$ which, in degree $n$, sends each $n$-cell of $|\cnerve\boxobj{m}|$ to an element in $(\Sigma^n\pi)_n=\pi$ with norm no greater than $i$.

\begin{lem}
  \label{lem:prerepresentation}
  There exists an isomorphism
  \begin{equation}
	\label{eqn:prerepresentation}
    \boundedcohomology{*}(B^\sharp;\pi)\cong\colim_{i\ra\infty}\left[B,C(\pi,n)^{(i)}\right]
  \end{equation}
  natural in seminormed Abelian groups $\pi$.
\end{lem}
\begin{proof}
  Let $D(\pi,n)$ be the cubical Abelian group
  $$D(\pi,n)=\CONNECTIVECHAINCOMPLEXES(\chains_*\BOX[-],\Sigma^n\pi).$$
  
  Call a chain map from a cellular chain complex to $\Sigma^n\pi$ \textit{$i$-bounded} if it sends $n$-cells to elements in $(\Sigma^n\pi)_n=\pi$ with norm no greater than $i$.  
  Regard $D(\pi,n)$ as filtered so that the $m$-cubes in the $i$th stage of its filtration consists of all $i$-bounded chain maps $\chains_*\BOX\boxobj{m}\ra\Sigma^n\pi$.   

  The inclusion $|\BOX[1]\ira\cnerve[1]|$ of CW complexes admits a cellular retraction mapping vertices to vertices and $m$-cells onto the unique $1$-cell for each $m>0$.  
  It therefore follows that the inclusion $\chains_*(\BOX[1]\ira\cnerve[1])$ of chain complexes admits a retraction which degree-wise sends each cell in $|\cnerve[1]|$ to either a cell in $|\BOX[1]|$ or $0$. 
  Every $i$-bounded chain map from $\chains_*(A\otimes\BOX[1])=\chains_*A\otimes\chains_*\BOX[1]$ to $\Sigma^n\pi$ thus extends to an $i$-bounded chain map on $\chains_*(A\otimes\cnerve[1])=\chains_*A\otimes\chains_*\cnerve[1]$.  
  It therefore follows that there exist natural bijections
  \begin{align*}
           \CUBICAL\SETS((-)\otimes\BOX[m],(D(\pi,n)^{(i)})_\sharp)_0
	&\cong \CUBICAL\SETS((-\otimes\BOX[m])^\sharp,(D(\pi,n)^{(i)}))_0\\
	&\cong \CUBICAL\SETS((-)^\sharp\otimes\cnerve[m],(D(\pi,n)^{(i)}))_0\\
	&\cong \CUBICAL\SETS((-)^\sharp\otimes\BOX[m],(D(\pi,n)^{(i)}))_0
  \end{align*}
  natural in $\BOX$-objects of the form $[m]$ for $m=0,1$. 
  Thus there exists a natural bijection 
  \begin{equation}
	\label{eqn:homotopical.sharp.adjunction}
    \pi_0\CUBICAL\SETS(-,(D(\pi,n)^{(i)})_\sharp)\cong\pi_0\CUBICAL\SETS((-)^\sharp,D(\pi,n)^{(i)}).
  \end{equation}

  Let $Z^{(n,i)}(-;\pi)$ denote the endofunctor on $\CUBICAL\SETS$ defined by
  $$Z^{(n,i)}(-;\pi)=\CUBICAL\SETS(-,D(\pi,n)^{(i)}).$$
 
  Under the Dold-Kan correspondence, $\pi_0Z^{(n,i)}(A;\pi)$ corresponds to the the set of all $n$-cocycles $c$ of $\chains^*_{\ell_\infty}(A;\pi)$ with $\sup_{\sigma\in A_n}\|c(\sigma)\|\leqslant i$ up to the smallest equivalence relating two such cocycles if their difference is the coboundary of an $(n-1)$-cochain $h$ of $\chains^*_{\ell_\infty}(A;\pi)$ with $\sup_{\sigma\in A_{n-1}}\|h(\sigma)\|\leqslant i$. 
  
  There exist natural isomorphisms
  \begin{align}
	  \colim_i[B,C(\pi,n)^{(i)}]
	  &\cong\label{eqn:spectrum.definition} \colim_i[B,(D(\pi,n)^{(i)})_{\sharp}]\\
	  &\cong\label{eqn:inverse.sd} \colim_{i,k}\pi_0\CUBICAL\SETS(\sd^kB,(D(\pi,n)^{(i)}_{\sharp})^{\sharp})\\
	  &\cong\label{eqn:spectrum.deformation} \colim_{i,k}\pi_0\CUBICAL\SETS(\sd^kB^\sharp,D(\pi,n)^{(i)}_{\sharp})\\
	  &\cong\label{eqn:adjunction} \colim_{i,k}\pi_0Z^{(n,i)}(\sd^kB^\sharp;\pi)\\
	  &\cong\label{eqn:dold-kan} \colim_{k}\boundedcohomology{n}(\sd^kB^\sharp;\pi)\\
	  &\cong\label{eqn:sd-invariance} \boundedcohomology{n}(B^\sharp;\pi),
  \end{align}
  where the colimits over $k\ra\infty$ are taken over all maps induced by $\gamma^{\mins k}$, because: (\ref{eqn:spectrum.definition}) by $C(\pi,n)^{(i)}=(D(\pi,n)^{(i)})_\sharp$; (\ref{eqn:inverse.sd}) by a formula for the hom-sets $[-,-]$ [Proposition \ref{prop:lipschitz.weak.homotopy.classes}]; (\ref{eqn:spectrum.deformation}) by the cubical homotopy equivalence $(D(\pi,n)^{(i)})_\sharp\simeq((D(\pi,n)^{(i)})_\sharp)^\sharp$ [Lemma \ref{lem:sharp.whitehead}]; (\ref{eqn:adjunction}) by (\ref{eqn:homotopical.sharp.adjunction}); (\ref{eqn:dold-kan}) by the Dold-Kan Correspondence; and (\ref{eqn:sd-invariance}) by subdivision invariance of $\ell_p$ cohomology [Lemma \ref{lem:sd.invariance}].
\end{proof}

Let $C_\infty(\pi,n)$ denote the coequalizer of the diagram
$$
\xymatrix{\coprod_{i}C(\pi,n)^{(i)}
  \ar@<.7ex>[rrrrrrr]^{(\amalg_iC(\pi,n)^{(i)})\otimes\cnerve\delta_-}
  \ar@<-.7ex>[rrrrrrr]_{\amalg_i(C(\pi,n)^{(i)}\ira C(\pi,n)^{(i+1)}\ira\amalg_jC(\pi,n)^{(j)})\otimes\cnerve\delta_+} 
  & 
  & 
  & 
  & 
  &
  & 
  & 
  \coprod_{i}C(\pi,n)^{(i)}\otimes\cnerve[1]}$$

\begin{eg}
  \label{eg:non-kan.BG}
  Consider a finitely generated group $G$ with word-length function
  $$\ell:G\ra\N.$$
  
  Fix a monotone function $\beta:\N\ra\N$.
  Define $\cnerve^\beta G$ as a mapping telescope like $C_\infty(\pi,n)$ based on a filtration of $\cnerve G$ whose $i$th stage is the maximal subpresheaf of $\cnerve G$ having as its $1$-cubes all $g\in G$ with $\ell(g)\leqslant\beta(i)$.  
  If $\beta(n)\ra\infty$, then the natural cubical function $\cnerve^\beta G\ra \cnerve G$ is a classical weak equivalence but generally not a Lipschitz weak equivalence.
  In the case $\Z$ is equipped with the standard absolute value word-length function and $\beta=\id_{\N}$, then $\cnerve^\beta\Z=C_\infty(\Z,1)$.
\end{eg}

The cubical set $C_\infty(\pi,n)$ admits the structure of a cubical semigroup whose multiplication is induced from dotted cubical functions making commutative diagrams of the form
\begin{equation*}
  \begin{tikzcd}
	    (C(\pi,n)^{(i)}\otimes\BOX[1])\times(C(\pi,n)^{(j)}\otimes \cnerve[1])\ar{r}[above]{\sigma}\ar[d,dotted] 
	  & C(\pi,n)^{(i)}\times C(\pi,n)^{(j)}\otimes \cnerve[1]^2\ar[d,equals]\\
        C(\pi,n)^{(i+j)}\otimes\BOX[1]
	  & C(\pi,n)^{(i)}\times C(\pi,n)^{(j)}\otimes \cnerve[1]^2
	    \ar{l}[below]{\mu\otimes \cnerve(\vee_{\boxobj{2}})}
  \end{tikzcd}
\end{equation*}
where $\sigma$ permutes factors and $\mu$ is induced by addition on $\pi$.
The map
$$C_\infty(\pi,n)\ra C(\pi,n)$$
induced by inclusions $C(\pi,n)^{(i)}\ira C(\pi,n)$ is a cubical semigroup homomorphism by the associativity of $\vee_{\boxobj{2}}$ and addition on $\pi$.

\begin{thm}
  \label{thm:cubical.cohomology.representable}
  \THMCubicalCohomologyRepresentable{}
\end{thm}
\begin{proof}
  Write $C_\infty(\pi,n)^{(j)}$ for the image of the canonical cubical function
  $$\coprod_{i=1}^j C(\pi,n)^{(i)}\otimes\cnerve[1]\ra C_\infty(\pi,n).$$

  Consider the following solid arrow diagram of Abelian groups:
  \begin{equation*}
	  \begin{tikzcd}
	      \boundedcohomology{n}(B;\pi)\ar[rr,dotted]
		&
		& {[B,C_\infty(\pi,n)]} 
		\\
		  \boundedcohomology{n}(B^\sharp;\pi)\ar{u}[left]{\boundedcohomology{n}(B\ira B^\sharp;\pi)}\ar[dotted]{r}[below]{\cong}
		& \colim_i[B,C(\pi,n)^{(i)}]\ar{r}[below]{\cong}
		& \colim_i[B,C_\infty(\pi,n)^{(i)}]\ar{u}[right]{\cong}
	  \end{tikzcd}
  \end{equation*}
  Kan cubical sets are sharp by $C\ira C^{\sharp}$ a weak equivalence from a fibrant object in the test model structure [Lemma \ref{lem:extensions}].  
  Therefore the left vertical arrow is a natural isomorphism [Lemma \ref{lem:prerepresentation}].
  There exists a bottom horizontal natural dotted isomorphism [Lemma \ref{lem:prerepresentation}].  
  The bottom solid horizontal arrow is an isomorphism because inclusion $C(\pi,n)^{(i)}\ira C_\infty(\pi,n)^{(i)}$ is a cubical homotopy equivalence and hence a Lipschitz weak equivalence.
  The right vertical arrow is an isomorphism because $\geometricrealization{B}$ and hence $\geometricrealization{\sd^kB}\times\geometricrealization{\BOX\boxobj{n}}$ has finite diameter [Lemma \ref{lem:bounded.nerve}], hence every Lipschitz map $\geometricrealization{\sd^kB}\times\geometricrealization{\BOX\boxobj{n}}\ra\geometricrealization{C_\infty(\pi,n)}$ factors through $\geometricrealization{C_{\infty}(\pi,n)^{(i)}}$ for $i\gg 0$, and hence every cubical function of the form $\sd^k(B\otimes\BOX\boxobj{n})\ra C_\infty(\pi,n)$ factors through $C_{\infty}(\pi,n)^{(i)}$ for $i\gg 0$.
\end{proof}

\begin{eg}
  \label{eg:non-kan.BG.cocycle}
  The first bounded cohomology of a group always vanishes.
  However, we can define an alternative cohomology theory $H^n_\beta(-;\pi)$, with respect to a \textit{bounding function} $\beta$, on groups $G$ equipped with word-length functions, by the rule 
  $$H^n_\beta(G;\pi)=[\cnerve^\beta G,C_\infty(\pi,n)],$$
  where $\cnerve^\beta G$ is a certain surgical construction on the cubical nerve $\cnerve G$ [Example \ref{eg:non-kan.BG}].
  The group $H^n_\beta(G;\pi)$ is intepretable as ordinary group cohomology, but where the growth of the cochains as word lengths increase is constrained by the bounding function $\beta$.  
  In the case $G=\Z$ is equipped its usual absolute value norm and $\beta=\id_{\N}$, $H^1_\beta(\Z;\Z)\neq 0$.
\end{eg}

\subsection{Uniform}
Let $K_\infty(\pi,n),K(\pi,n)$ denote the uniform realizations of the respective cubical sets $C_\infty(\pi,n),C(\pi,n)$.  
The semigroup multiplications induce a multiplication turning the induced uniform map $K_\infty(\pi,n)\ra K(\pi,n)$ into a homomorphism of commutative semigroup objects in $\UNIFORM$.

\begin{eg}
  \label{eg:bounded.singular.cohomology}
  There exists a natural identification
  $$\boundedcohomology{*}(\sing_{\DEL_{\pls}}X;\R)=\boundedcohomology{*}(\sing_{\BOX_1^*}X;\R),$$
  where $\sing_{\DEL_{\pls}}$ and $\sing_{\BOX_1^*}$ are the right adjoint to topological realization functors from, respectively, semi-simplicial sets and classically defined cubical sets, and the cochain complexes used to define the cohomologies are defined as the bounded subcomplexes of the usual cochain complexes for semi-simplicial and classically defined cubical sets \cite[Example 3.2, Proposition 4.3]{loeh2015cubical}.
  The same kind of reasoning used to give this kind of reasoning, the method of acyclic models \cite{eilenberg1953acyclic}, shows that we can make the identification $\boundedcohomology{*}(\sing_{\DEL_+}X;\R)=\boundedcohomology{*}(\sing_{\BOX}X;\R)$.
\end{eg}

\begin{cor}
  \label{cor:singular.cohomology.representable}
  \CORSingularCohomologyRepresentable{}
\end{cor}

\subsection{Comparisons}
For each $n$, there exists an isomorphism 
$$H^*(B;\pi)\cong H^*(|B|;\pi)$$
natural in cubical sets $B$ and Abelian groups $\pi$, because $\chains^*(B,\pi)$ is the cellular cochain complex of $|B|$ when $\pi$ is trivially seminormed.  
A special case of the embedding of the Lipschitz homotopy category into the uniform homotopy category [Corollary \ref{cor:embedding}] is the following refinement of that isomorphism.

\begin{prop}
  \label{prop:cohomological.comparison}
  Uniform realization induces an isomorphism
  $$\left[B,C_\infty(\pi,n)\right]\cong \left[\uniformrealization{B},K_\infty(\pi,n)\right]$$
  of Abelian groups natural in cubical sets $B$ and seminormed Abelian groups $\pi$.
\end{prop}

\section{Conclusion}
Lipschitz homotopy theory naturally extends to the model categorical setting.  
Consider a closed Cartesian monoidal presheaf model category $\modelcat{1}$ in which all objects are cofibrant.
Further suppose there exists a distinguished \textit{generating diagram} $J$ for the acyclic cofibrations, a diagram in the arrow category of $\modelcat{1}$ for which the fibrations in $\modelcat{1}$ are characterized as having suitably natural right lifting properties against the objects in $J$ \cite{barthel2013construction}.
The choice of $J$ determines a kind of metric structure on $\modelcat{1}$ as follows.  
An algebraic small object argument builds a fibrant replacement monad $R$ for $\modelcat{1}$ from $J$ as a transfinite colimit of pointed endofunctors.
A weak equivalence in $\modelcat{1}$ is exactly a homotopy equivalence after applying $R$.  
We can call a weak equivalence $\zeta$ in $\modelcat{1}$ a \textit{Lipschitz weak equivalence} if for each $\modelcat{1}$-object $o$, $\zeta^o$ is a homotopy equivalence after applying one of the finite intermediate stages in the construction of $R$.  
In the case $\modelcat{1}$ is a category of presheaves over a test category \cite{cisinskiprefaisceaux,grothendieck1983pursuing}, a natural question is whether there exists a choice of generating diagram $J$ so that the localization by the resulting class of Lipschitz weak equivalences is equivalent to the Lipschitz homotopy category of cubical sets.
The special case where $\modelcat{1}$ is the category of simplicial sets, each of which is plausibly sharp in a Lipschitz-theoretic sense (eg. \cite{rezk1998fibrations}), is of particular interest in simplicial $\ell_p$ cohomology.

\section{Acknowledgements}
The first author was partially supported by AFOSR grant FA9550-16-1-0212.

\addcontentsline{toc}{section}{Appendix}
\addtocontents{toc}{\protect\setcounter{tocdepth}{0}}
\appendix

\section{Simplicial sets}\label{sec:simplicial}
We factor both geometric realization and sharp replacement [Definition \ref{defn:sharp.replacement}] through simplicial sets.  
These factorizations make it easy to formally construct certain maps needed in proofs throughout the paper.
Let $\DEL$ denote the category of non-empty finite ordinals $[0],[1],\ldots$ and all monotone functions between them.
Let $\SIMPLICIAL\SETS$ denote the category of functors $\OP\DEL\ra\SETS$ and natural transformations between them.  
\textit{Simplicial sets} and \textit{simplicial functions} are the objects and morphisms of $\SIMPLICIAL\SETS$. 
Write $\snerve\cat{1}$ for the \textit{simplicial nerve} natural in a small category $\cat{1}$, defined as naturally sending each $\DEL$-object $[n]$ to the set of functors $[n]\ra\cat{1}$.  
Let $\tri\dashv\qua$ denote the monoidal adjunction
$$\tri:\CUBICAL\SETS\lras\SIMPLICIAL\SETS:\qua$$
with $\tri$ naturally sending a representable $\BOX\boxobj{n}$ to the simplicial nerve $\snerve\,\boxobj{n}$.
Formally, $\tri$ is cocontinuous.  
Let $\qua$ denote the right adjoint to $\tri$.  
Less formally, the monad $\qua\,\tri$ is cocontinuous.

\begin{lem}
  \label{lem:qt.cocontinuous}
  There exists an isomorphism
  $$C^{\sharp}\cong\qua\,\tri\,C$$
  natural in cubical sets $C$.  
\end{lem}

A proof for the corresponding statement under a different definition of the $\BOX$ category \cite[Lemma 4.2]{krishnan2015cubical} only depends on the following facts: every $\BOX$-morphism uniquely factors into a surjective $\BOX$-morphism followed by an injective $\BOX$-morphism; and injective $\BOX$-morphisms are uniquely determined by where they send extrema.  
Therefore a proof for the lemma above is omitted.

\begin{lem}
  \label{lem:simplicial.geometric.realization}
  Geometric realization factors through triangulation.  
\end{lem}
\begin{proof}
  Let $F$ be the cocontinuous functor 
  $$F:\SIMPLICIAL\SETS\ra\METRIC$$
  sending each representable $\DEL[n]$ to the subset of $\R^n$ consisting of all tuples $(x_1,\cdots,x_n)$ with $1\geqslant x_1\geqslant \cdots\geqslant x_n\geqslant 0$, equipped with the metric such that two points $x,y$ are $\delta$-close if $x_i,y_i$ are $\delta$-close for all $i$, and sending each simplicial function $\DEL[\phi]:\DEL[m]\ra\DEL[n]$ to the linear extension of the function sending the point $(1,1,\cdots,1,0,\cdots,0)$ with $m-i$ non-zero coordinates to the point $(1,1,\cdots,1,0,\cdots,0)$ with $n-\phi(0)$ non-zero entries.   
  It suffices to show $F\tri\,\BOX[-]$ is a monoidal functor $\BOX\ra\METRIC$.
  For then there would exist $\METRIC$-isomorphisms
  \begin{align*}
    F\tri\,\BOX\boxobj{n}
	&\cong F(\tri\,\BOX[1])^n\\
	&\cong F(\DEL[1])^n\\
	&=\geometricrealization{\BOX[1]}^n\\
	&=\geometricrealization{\BOX[1]^{\boxobj{n}}}\\
  \end{align*}
  natural in $\BOX$-object $\boxobj{n}$.
  Then $F\tri=\geometricrealization{-}$ would follow because all functors in sight cocontinuous.

  Let $\pi_{i}$ denote projection $\boxobj{n}\ra[1]$ onto the $i$th factor.  
  Let $\phi$ denote an injective monotone function $[n]\ra\boxobj{n}$.  
  It suffices to show that the $1$-Lipschitz function $\prod_iF\snerve\pi_i\phi$ for each $\phi$, and hence the $1$-Lipschitz bijection $\prod_i(F\,\snerve\pi_i):F\tri\,\BOX\boxobj{n}\ra(F\tri\,\BOX[1])^n$, is $1$-bi-Lipschitz.  

  Fix $\phi$.
  For each $z\in F\DEL[n]$, let $\bar{z}_i=(F\snerve\pi_{i}\phi)(z)$.  
  Consider $x,y\in F\DEL[n]$ for which $\prod_i\bar{x}_i,\prod_i\bar{y}_i$ are $\epsilon$-close.
  Then $\bar{x}_i,\bar{y}_i$ are $\epsilon$-close for each $i$.  
  For each $1\leqslant i\leqslant n$, $\pi_i\phi:[n]\ra[1]$ is the monotone function sending $i-1$ to $0$ and $i$ to $1$.  
  Therefore $\bar{x}_i=x_i$ and $\bar{y}_i=y_i$ for each $1\leqslant i\leqslant n$.  
  Hence $x,y$ are $\epsilon$-close in $F\DEL[n]$.  
\end{proof}

\begin{rem}
  \label{rem:simplicial.disadvantages}
  Geometric realization thus extends along $\tri$ to a functor 
  \begin{equation}
	\label{eqn:simplicial.geometric.realization}
    \SIMPLICIAL\SETS\ra\METRIC
  \end{equation}
  This geometric realization functor does not commute with abstract barycentric subdivision up to natural homeomorphism.
  There does exist a homeomorphism, linear in barycentric coordinates, between a topological simplicial complex and its topological barycentric subdivision.  
  There even exists a homeomorphism, quadratic in barycentric coordinates, between more general geometric realizations of a simplicial set and its abstract barycentric subdivision \cite{fritsch1967homoomorphie}.  
  However, neither of these homeomorphisms are bi-Lipschitz for spaces of unbounded dimension.  
  Geometric realization (\ref{eqn:simplicial.geometric.realization}) does commute, up to bi-Lipschitz equivalence linear in barycentric coordinates, with abstract edgewise subdivision \cite{segal1973configuration}.
  But iterated edgewise subdivision, unlike iterated baryentric subdivision, does not factor even locally through a polyhedral complex. 
\end{rem}

\begin{proof}[proof of Lemma \ref{lem:sharp.sharp}]
  The inclusion $C^\sharp\ira C^{\sharp\sharp}$ can be identified with
  $$\eta_{\qua\,\tri\,C}:\qua\,\tri\,C\ira \qua\,\tri\,\qua\,\tri\,C,$$
  where $\eta$ is the unit of the adjunction $\tri\dashv\qua$ [Lemma \ref{lem:qt.cocontinuous}], and therefore admits a retraction $\qua\,\epsilon_{\tri\,C}$ where $\epsilon$ is the counit of the adjunction $\tri\dashv\qua$.
\end{proof}

\begin{proof}[proof of Lemma \ref{lem:geometric.sharp.retraction}]
  The simplicial function
  $$\tri\,(C\ira C^{\sharp})=\tri\,\eta_C:C\ira\tri\,\qua\,\tri\,C$$
  [Lemma \ref{lem:qt.cocontinuous}] admits a retraction natural in cubical sets $C$ by the zig-zag identities for the adjunction $\tri\dashv\qua$.
  Therefore $\geometricrealization{C\ira C^{\sharp}}$ admits a $1$-Lipschitz retraction $r_C$ natural in cubical sets $C$ [Lemma \ref{lem:simplicial.geometric.realization}].  
\end{proof}

\section{Categories of (co)fibrant objects}\label{sec:fibrant.objects}
We briefly recall definitions associated to categories of (co)fibrant objects and the basic fact that such categories admit localizations by their weak equivalences.  
The reader is referred elsewhere \cite{brown1973abstract} for more details, including concrete constructions of morphisms in the homotopy category and applications to the construction of cocycles in cohomology.  
A \textit{category of fibrant objects} is a category $\homotopicalcat{1}$ with weak equivalences together with a class of morphisms, called the \textit{weak equivalences}, and a class of morphisms, called the \textit{fibrations}, such that the following conditions hold:
\begin{enumerate}
  \item $\homotopicalcat{1}$ has all finite limits 
  \item The weak equivalences in $\homotopicalcat{1}$ satisfy the 2-out-of-3 axiom.
  \item Each diagonal map in $\homotopicalcat{1}$ factors as a weak equivalence followed by a fibration.
  \item Every terminal $\homotopicalcat{1}$-morphism is a fibration.
  \item Fibrations in $\homotopicalcat{1}$ are closed under pullback.
  \item The $\homotopicalcat{1}$-morphisms that are at once fibrations and weak equivalences are closed under pullbacks.
\end{enumerate}

Consider a category $\homotopicalcat{1}$ of fibrant objects. 
Call a $\homotopicalcat{1}$-morphism $\zeta$ an \textit{acyclic fibration} if $\zeta$ is at once a fibration and weak equivalence.
A $\homotopicalcat{1}$-morphism is a \textit{cofibration} if it has the left lifting property against all acyclic fibrations.
Suppose further that $\homotopicalcat{1}$ has an initial object. 
Then the \textit{cofibrant objects} are those $\homotopicalcat{1}$-objects $o$ for which every $\homotopicalcat{1}$-morphism of the form $o\ra x$ lifts against every acyclic fibration to $x$.

\begin{eg}
  A model category with all objects are fibrant is a category of fibrant objects.
\end{eg}

Every category $\homotopicalcat{1}$ of fibrant objects admits a localization
$$\homotopicalcat{1}\ra h\homotopicalcat{1}$$
by the weak equivalences.  
Under some mild assumptions, the h-equivalences and h-fibrations form a model structure \cite{barthel2013construction}.  
Under no assumptions, the h-equivalences and h-fibrations form a category of fibrant objects.

\begin{eg}
  \label{eg:h-fibrant}
  A bicomplete $\TOP$-category is a category of fibrant objects in which \ldots
  \begin{enumerate}
	\item weak equivalences are the h-equivalences
	\item fibrations are the h-fibrations
  \end{enumerate}
  The verification of all but the factorization axiom is formal.  
  The factorization axiom essentially follows from observations elsewhere \cite[Lemma 4.3.1 (iii)]{may2006parametrized} but for the inconsequential detail that $\TOP$, while complete, cocomplete, and Cartesian closed, is not the usual category of weak Hausdorff k-spaces.    
\end{eg}

A \textit{category of cofibrant objects} is a category $\homotopicalcat{1}$ together with a class of morphisms, called the \textit{weak equivalences}, and a class of morphisms, called the \textit{cofibrations}, such that $\OP{\homotopicalcat{1}}$ is a category of fibrant objects with the same weak equivalences but where the fibrations are the cofibrations in $\homotopicalcat{1}$.  
Consider a category $\homotopicalcat{1}$ of cofibrant objects.  
Call a $\homotopicalcat{1}$-morphism $\zeta$ an \textit{acyclic cofibration} if $\zeta$ is at once a cofibration and weak equivalence.
Suppose further that $\homotopicalcat{1}$ has a terminal object. 
Then the \textit{fibrant objects} in $\homotopicalcat{1}$ are those $\homotopicalcat{1}$-objects $o$ for which every $\homotopicalcat{1}$-morphism to $o$ extends along every acyclic cofibration.
By duality, every category $\homotopicalcat{1}$ of cofibrant objects admits a localization
$$\homotopicalcat{1}\ra h\homotopicalcat{1}$$
by the weak equivalences.  

\bibliography{at,cubical,cohomology,topology,ditopology}{}

\begin{thebibliography}{10}

\bibitem{andre2012uniform}
Yves Andr{\'e}.
\newblock {U}niform {S}heaves and {D}ifferential {E}quations.
\newblock {\em Rendiconti del Seminario Matematico della Universit{\`a} di
  Padova}, 128:345--372, 2012.

\bibitem{bahauddin1973homology}
Mohammed Bahauddin and John Thomas.
\newblock The {H}omology of {U}niform {s}paces.
\newblock {\em Canadian Journal of Mathematics}, 25(3):449--455, 1973.

\bibitem{barthel2013construction}
Tobias Barthel and Emily Riehl.
\newblock On the {C}onstruction of {F}unctorial {F}actorizations for {M}odel
  {C}ategories.
\newblock {\em Algebraic \& Geometric Topology}, 13(2):1089--1124, 2013.

\bibitem{brown1973abstract}
Kenneth~S Brown.
\newblock Abstract {H}omotopy {T}heory and {G}eneralized {S}heaf {C}ohomology.
\newblock {\em Transactions of the American Mathematical Society},
  186:419--458, 1973.

\bibitem{huebschmann2012ronald}
Ronald Brown, Philip~J. Higgins, and Rafael Sivera.
\newblock {\em {N}onabelian {A}lgebraic {T}opology}, volume~15 of {\em EMS
  Tracts in Mathematics}.
\newblock European Mathematical Society (EMS), Z\"{u}rich, 2011.
\newblock Filtered spaces, crossed complexes, cubical homotopy groupoids, With
  contributions by Christopher D. Wensley and Sergei V. Soloviev.

\bibitem{cisinskiprefaisceaux}
Denis-Charles Cisinski.
\newblock Les {P}r\'{e}faisceaux comme {M}od\`eles des {T}ypes d'{H}omotopie.
\newblock {\em Ast\'{e}risque}, (308):xxiv+390, 2006.

\bibitem{cole2006mixing}
Michael Cole.
\newblock Mixing {M}odel {S}tructures.
\newblock {\em Topology and its Applications}, 153(7):1016--1032, 2006.

\bibitem{connes1988hyperbolic}
Alain Connes and Henri Moscovici.
\newblock Conjecture de {N}ovikov et {G}roupes {H}yperboliques.
\newblock {\em C. R. Acad. Sci. Paris S\'{e}r. I Math.}, 307(9):475--480, 1988.

\bibitem{curtis1971simplicial}
Edward~B Curtis.
\newblock Simplicial {H}omotopy {T}heory.
\newblock {\em Advances in Mathematics}, 6(2):107--209, 1971.

\bibitem{eilenberg1953acyclic}
Samuel Eilenberg and Saunders MacLane.
\newblock {A}cyclic {M}odels.
\newblock {\em American journal of mathematics}, 75(1):189--199, 1953.

\bibitem{elek1998coarse}
G{\'a}bor Elek.
\newblock Coarse {C}ohomology and {lp}-{C}ohomology.
\newblock {\em K-theory}, 13(1):1--22, 1998.

\bibitem{fritsch1967homoomorphie}
Dieter Fritsch, Rudolf et~Doll.
\newblock L'{H}om\'{e}omorphisme des {R}\'{e}alisations {G}\`{e}om\'{e}triques
  d'un {E}nsemble {S}emi-{S}implicial et sa {S}ubdivision {N}ormale.
\newblock {\em archive de mathématiques}, 18(5):508--512, 1967.

\bibitem{gersten1996note}
SM~Gersten.
\newblock A {N}ote on {C}ohomological {V}anishing and the {L}inear
  {I}soperimetric {I}nequality.
\newblock {\em preprint available at http://www. math. utah. edu/~ gersten},
  1996.

\bibitem{gol1988rham}
VM~Gol'dshtein, VI~Kuz'minov, and IA~Shvedov.
\newblock De {R}ham {I}somorphism of the {lp}-{C}ohomology of {N}oncompact
  {R}iemannian {M}anifolds.
\newblock {\em Siberian Mathematical Journal}, 29(2):190--197, 1988.

\bibitem{goubault2019directed}
Eric Goubault and Samuel Mimram.
\newblock {D}irected {H}omotopy in {N}on-{P}ositively {C}urved {S}paces.
\newblock {\em arXiv preprint arXiv:1908.06684}, 2019.

\bibitem{grandis2003cubical}
Marco Grandis and Luca Mauri.
\newblock Cubical sets and their site.
\newblock {\em Theory Appl. Categ}, 11(8):185--211, 2003.

\bibitem{gromov1982volume}
Michael Gromov.
\newblock Volume and {B}ounded {C}ohomology.
\newblock {\em Publications Math{\'e}matiques de l'IH{\'E}S}, 56:5--99, 1982.

\bibitem{grothendieck1983pursuing}
Alexander Grothendieck.
\newblock {\em Pursuing stacks}.
\newblock unpublished manuscript, 1983.

\bibitem{hatcher2005algebraic}
Allen Hatcher.
\newblock {\em {A}lgebraic {T}opology}.
\newblock Cambridge University Press, Cambridge, 2002.

\bibitem{isbell1964uniform}
John~Rolfe Isbell.
\newblock {\em {U}niform {S}paces}.
\newblock Number~12. American Mathematical Soc., 1964.

\bibitem{jardine2004simplicial}
J.~F. Jardine.
\newblock Simplicial {A}pproximation.
\newblock {\em Theory Appl. Categ.}, 12:No. 2, 34--72, 2004.

\bibitem{jardine2002cubical}
John~F Jardine.
\newblock Cubical {H}omotopy {T}heory: A {B}eginning.
\newblock {\em preprint}, 2002.

\bibitem{kan1957css}
Daniel~M Kan.
\newblock On {CSS} {C}omplexes.
\newblock {\em American Journal of Mathematics}, 79(3):449--476, 1957.

\bibitem{kapulkin2019co}
Chris Kapulkin, Zachery Lindsey, and Liang~Ze Wong.
\newblock A {C}o-reflection of {C}ubical {S}ets into {S}implicial {S}ets with
  {A}pplications to {M}odel {S}tructures.
\newblock {\em arXiv preprint arXiv:1906.09203}, 2019.

\bibitem{krishnan2015cubical}
Sanjeevi Krishnan.
\newblock Cubical {A}pproximation for {D}irected {T}opology {I}.
\newblock {\em Applied Categorical Structures}, 23(2):177--214, 2015.

\bibitem{lawvere1973metric}
F~William Lawvere.
\newblock Metric {S}paces, {G}eneralized {L}ogic, and {C}losed {C}ategories.
\newblock {\em Rendiconti del seminario mat{\'e}matico e fisico di Milano},
  43(1):135--166, 1973.

\bibitem{loeh2015cubical}
Clara Loeh, Cristina Pagliantini, and Sebastian Waeber.
\newblock Cubical {S}implicial {V}olume of 3-manifolds.
\newblock {\em arXiv preprint arXiv:1508.03017}, 2015.

\bibitem{maltsiniotis2009categorie}
Georges Maltsiniotis et~al.
\newblock La {C}at{\'e}gorie {C}ubique avec {C}onnexions est une
  {C}at{\'e}gorie {T}est {S}tricte.
\newblock {\em Homology, Homotopy and Applications}, 11(2):309--326, 2009.

\bibitem{may2006parametrized}
J~Peter May and Johann Sigurdsson.
\newblock {\em Parametrized {H}omotopy {T}heory}.
\newblock Number 132. American Mathematical Soc., 2006.

\bibitem{mineyev2002bounded}
Igor Mineyev.
\newblock Bounded {C}ohomology {C}haracterizes {H}yperbolic {G}roups.
\newblock {\em The Quarterly Journal of Mathematics}, 53(1):59--73, 2002.

\bibitem{ogle2010polynomially}
Crichton Ogle.
\newblock Polynomially {B}ounded {C}ohomology and the {N}ovikov {C}onjecture.
\newblock {\em arXiv preprint arXiv:1004.4680}, 2010.

\bibitem{pansu2007lp}
Pierre Pansu.
\newblock {lp}-{C}ohomology of {S}ymmetric {S}paces.
\newblock {\em arXiv preprint math/0701151}, 2007.

\bibitem{quillen2006homotopical}
Daniel~G. Quillen.
\newblock {\em {H}omotopical {A}lgebra}.
\newblock Lecture Notes in Mathematics, No. 43. Springer-Verlag, Berlin-New
  York, 1967.

\bibitem{MR3010108}
Bobby Ramsey.
\newblock A {S}pectral {S}equence for {P}olynomially {B}ounded {C}ohomology.
\newblock {\em J. Pure Appl. Algebra}, 217(6):1153--1163, 2013.

\bibitem{rezk1998fibrations}
Charles Rezk.
\newblock Fibrations and {H}omotopy {C}olimits of {S}implicial {S}heaves.
\newblock {\em arXiv preprint math/9811038}, 1998.

\bibitem{rice1983cartesian}
MD~Rice and GJ~Tashjian.
\newblock Cartesian-closed {C}oreflective {S}ubcategories of {U}niform {S}paces
  {G}enerated by {C}lasses of {M}etric {S}paces.
\newblock {\em Topology and its Applications}, 15(3):301--312, 1983.

\bibitem{segal1973configuration}
Graeme Segal.
\newblock {C}onfiguration-spaces and {I}terated {L}oop-spaces.
\newblock {\em Inventiones mathematicae}, 21(3):213--221, 1973.

\bibitem{steenrod1967convenient}
Norman~E Steenrod.
\newblock A {C}onvenient {C}ategory of {T}opological {S}paces.
\newblock {\em Michigan Mathematical Journal}, 14(2):133--152, 1967.

\bibitem{tonks1992cubical}
Andrew~P Tonks.
\newblock {C}ubical {G}roups which are {K}an.
\newblock {\em Journal of pure and applied algebra}, 81(1):83--87, 1992.

\bibitem{wyler1973convenient}
Oswald Wyler.
\newblock {C}onvenient {C}ategories for {T}opology.
\newblock {\em General Topology and its Applications}, 3(3):225--242, 1973.

\end{thebibliography}
\bibliographystyle{plain}

\end{document}